\documentclass[11pt]{amsart}
\usepackage[utf8]{inputenc}
\usepackage[normalem]{ulem}

\usepackage{enumitem}

\title{}

\usepackage{epsfig,color,mathrsfs,amsbsy,amsfonts,amsmath,amssymb,amsthm,fullpage, hyperref,cancel}

\newcounter{thm}
\newtheorem{theorem}[thm]{Theorem}
\newtheorem{definition}{Definition}

\newtheorem{prop}{Proposition}
\newtheorem{cor}{Corollary}
\newtheorem{lemma}{Lemma}

\newcounter{example}
\newenvironment{example}[1][]{\refstepcounter{example}\par\medskip
	\noindent \textbf{Example~\theexample. #1} \rmfamily}{\medskip}
\newcounter{remark}
\newenvironment{remark}[1][]{\refstepcounter{remark}\par\medskip
	\noindent \textbf{Remark~\theremark. #1} \rmfamily}{\medskip}

\def \Frob{\text{Frob}}
\def \ord{\text{ord}}
\def\ff{\mathfrak f}
\def \fh{\mathfrak h}
\def\fg{\mathfrak g}

\def\F{\mathbb F}
\def\C{\mathbb C}

\def\P{\mathbb P}
\def\Q{\mathbb Q}
\def\R{\mathbb R}

\def\Z{\mathbb Z}

\def\eps{\varepsilon}
\def \l {\lambda}
\def\ol{\overline}
\def\CC#1#2{\binom {#1}{#2}}
\def\({\left(}
\def\){\right)}
\def \f#1#2{\frac{#1}{#2}}
\def\G{\Gamma}

\newcommand*\HYPERskip{&}
\catcode`,\active
\newcommand*\pFq{
	\begingroup
	\catcode`\,\active
	\def ,{\HYPERskip}%
	\doHyper
}
\catcode`\,12
\def\doHyper#1#2#3#4#5{%
	\, _{#1}F_{#2}\left[\begin{matrix}#3 \smallskip \\  #4\end{matrix} \; ; \; #5\right]%
	\endgroup
}
\def\M#1#2#3#4{\begin{pmatrix}#1&#2\\#3&#4\end{pmatrix}}

\catcode`,\active

\catcode`\,12

\catcode`,\active

\catcode`\,12

\catcode`,\active

\catcode`\,12

\newcommand*\HYPERpp{&}
\catcode`,\active
\newcommand*\pPPq{
	\begingroup
	\catcode`\,\active
	\def ,{\HYPERpp}%
	\doHyperFpp
}
\catcode`\,12
\def\doHyperFpp#1#2#3#4#5{%
	\, _{#1}{\mathbb P}_{#2}\left[\begin{matrix}#3 \smallskip \\  #4\end{matrix} \; ; \; #5\right]%
	\endgroup
}

\def \bk {\color{black}}
\def \P{\mathbb P}
\def \lc {\left \{}
\def \rc {\right\}}
\def \Tr {\text{Tr}}
\begin{document}
	\title{A Whipple $_7F_6$ formula revisited}
	\author{Wen-Ching Winnie Li$^\dag$, Ling Long$^\ddag$ and Fang-Ting Tu$^\ddag$}
	
	\begin{abstract}
		A well-known formula of Whipple relates certain hypergeometric values $_7F_6(1)$ and $_4F_3(1)$. In this paper we revisit this relation from the viewpoint of the underlying hypergeometric data $HD$, to which there are also associated hypergeometric character sums and Galois representations. We explain a special structure behind Whipple's formula when the hypergeometric data $HD$ are primitive and  self-dual.  If the data are also defined over $\Q$, by the work of Katz, Beukers, Cohen, and Mellit, there are compatible families of $\ell$-adic representations of the absolute Galois group of $\Q$ attached to $HD$. For specialized choices of $HD$, these Galois representations are shown to be decomposable and automorphic. As a consequence, the values of the corresponding hypergeometric character sums can be explicitly expressed in terms of Fourier coefficients of certain modular forms.  We further relate the  hypergeometric values $_7F_6(1)$ in Whipple's formula to the  periods of these modular forms.  
	\end{abstract}
	\maketitle
	\section{Introduction}
	
	\subsection{Hypergeometric data}
	
	Hypergeometric functions have a long history in Mathematics, and they occur in many different research branches. In recent years, rapid progress has been made in hypergeometric character sums under the framework of hypergeometric motives. It gives a unifying perspective to many hypergeometric structures and identities, see \cite[et al]{Beukers-matrix, BCM, DPVZ,Win3a,Win3X, Greene,Hoffman-Tu, Katz, Long18,McCarthy, RV}. Classical hypergeometric formulas such as Clausen's formula have  a finite field analogue by Evans and Greene  \cite{Evans-Greene}. McCarthy \cite{McCarthy} obtained several finite field analogues of well-posed hypergeometric values with argument being $\pm 1$. Based on this work,   McCarthy and Papanikolas  relate hypergeometric sums with Siegel eigenforms in \cite{McCarthy-Papanikolas}.  Demb\'el\'e, Panchishkin, Voight, and Zudilin in \cite{DPVZ} provided numerical evidence relating special values of certain hypergeometric L-functions and  Asai L-functions of Hilbert modular forms.
	In \cite{Win3X}, the authors 
	interpreted hypergeometric character sums in a manner parallel to the development of classical hypergeometric functions. 
	
	Given a  hypergeometric datum $HD=\{\alpha=\{a_1,a_2,\cdots,a_n\}, \beta=\{1,b_2,\cdots,b_n\};\l\}$ with $a_i, b_j \in \Q$,  satisfying $a_i-b_j\notin \Z$ for all $i,j$ (in which case the pair $\alpha, \beta$ is called {\it primitive}), and nonzero $\l$ in a cyclotomic field,\footnote{As an element in a global field, $\l$ in different settings below is suitably interpreted: in (1) $\l$ is viewed as a complex number via an embedding of the cyclotomic field in $\C$, while in (2) $\l$ is viewed as an element in the residue field of a nonarchimedean place of the cyclotomic field, when appropriate.} 
	one can
	\begin{itemize}
		\item[(1)]  Define a classical (generalized) hypergeometric series (HGS) $F(\alpha, \beta;\l)$  by
		$$
		F(\alpha,\beta; \l):=\sum_{k=0}^\infty \frac{(a_1)_k(a_2)_k\cdots (a_n)_k}{(1)_k(b_2)_k\cdots (b_n)_k} \l^k
		$$ 
		if it converges, where $(a)_k=\G(a+k)/\G(a)=a(a+1)\cdots(a+k-1)$ is the Pochhammer symbol.  It satisfies a Fuchsian differential equation of order $n$ with singularities at $0$, $1$, and $\infty$.
		Solutions at nonsingular points form a local system and give rise to a monodromy representation of the fundamental group of  $\C P^1\smallsetminus \{0,1,\infty\}$ with a base point, {see \cite{Beukers-note, Katz, Yoshida} for example.}
		We use $F(\alpha, \beta;\l)_n:=\sum_{k=0}^n\frac{(a_1)_k(a_2)_k\cdots (a_n)_k}{(1)_k(b_2)_k\cdots (b_n)_k} \l^k$ to denote its truncation at $n$, which plays a vital role in Dwork's unit root theory for hypergeometric functions \cite{Dwork};  see \cite{BV1, BV2} by Beukers and Vlasenko for recent developments.
		
		\item[(2)] Compute hypergeometric character sums as hypergeometric functions over finite fields $\F_q$. In this paper we focus on the finite period function $\P(HD;q)$ in \cite{Win3X} and the finite hypergeometric function  $H_q(HD)$ in \cite{BCM}. 
		
		\item[(3)] Associate $\ell$-adic Galois representations, such as $\rho_{HD,\ell}$  in \cite{Katz, Katz09}  by Katz for the Galois group of a cyclotomic field,  
		and its extension $\rho_{HD,\ell}^{BCM}$ in \cite{BCM}  by Beukers, Cohen and Mellit  to a representation of the absolute Galois group of $\Q$, when possible.  
		The traces of these representations are essentially given by $\P(HD;q)$ (see \eqref{eq:P-by-induction} below) or $H_p(HD)$ (see \eqref{eq:hq} below), respectively. These are the $\ell$-adic counterparts of  (1).
	\end{itemize}
	These perspectives are interconnected and are compatible with each other. For an overview, see \cite{Hoffman-Tu} by Hoffman and Tu. As we will see from the elaboration below, (2) and (3) in the enumeration above only depend on the multi-sets $\alpha$ and $\beta$ mod $\Z$, while $\P(HD;q)$ and $\rho_{HD, \ell}$ also depend on the order of the elements in the multi-sets. 
	
	There are different versions of hypergeometric character sums. One of
	the most convenient choices for our purpose is the period function $_n\P_{n-1}$ (see \eqref{eq:P-by-induction} below)  given in \cite[\S4]{Win3X}.   
	For any finite field $\F_q$ of characteristic $p>2$, let $\omega$ be a generator of the group $\widehat{\F_q^\times}:=\text{Hom}(\F_q^\times, \C^\times)$   
	of multiplicative characters of $\F_q$. 
	Let $\alpha=\{a_1,\cdots,a_n\}, \beta=\{1,b_2,\cdots,b_n\}$ be two multi-sets with entries in $\Q$. Denote by $M:=lcd(\alpha\cup \beta)$ the positive least  common denominator of all $a_i,b_j$. 
	For a finite field $\F_q$ containing a primitive $M$th root of $1$ 
	and any $\l\in \F_q$, we write
	
	\begin{equation}\label{eq:P}
	\P(\alpha,\beta;\l;\F_q; \omega):=\pPPq{n}{n-1}{\omega^{(q-1)a_1}&\omega^{(q-1)a_2}&\cdots&\omega^{(q-1)a_n}}{&\omega^{(q-1)b_2}&\cdots&\omega^{(q-1)b_n}}{\l;q}.
	\end{equation}
	The value depends on $a_i$ and $b_j$ modulo $\Z$, as well as their  order. All characters $\omega^{(q-1)a_i}$ and $\omega^{(q-1)b_j}$ have orders dividing $M$, so the above $\P$ function depends only on $\omega^{(q-1)/M}$, which has order $M$. 
	When the pair $\{\alpha,\beta\}$ is primitive, that is, $a_i-b_j\notin \Z$ for all $i,j$, as shown in \cite{Win3X}, 
	the above $\P$ function can be normalized to rid of the dependence on the order
	of the $a_i$ and $b_j$.  For any number field $K$, denote by $G_K =\text{Gal}(\overline \Q/K)$ the absolute Galois group of $K$. When $K=\Q(\zeta_M)$ is the cyclotomic field obtained by adjoining a primitive $M$th root of unity $\zeta_M$, we further abbreviate $G_{\Q(\zeta_M)}$ by $G(M)$. Throughout this paper, a prime ideal will always be nonzero.
	For a prime ideal $\wp$ of  $\Z[\zeta_M,1/M]$,  
	$\zeta_M \mod \wp$ in the residue field $\kappa_{\wp}$ of $\wp$ has order $M$, and it generates the cyclic group $(\kappa_{\wp}^\times)^{(N(\wp)-1)/M}$.
	Put  
	$$\P(\alpha,\beta;\l;\kappa_\wp) = \P(\alpha,\beta;\l;\kappa_\wp; \omega_\wp)$$ where $\omega_{\wp}$ is a generator of $\widehat{\kappa_{\wp}^\times}$ so that $$\omega_\wp(\zeta_M \mod \wp) = \zeta_M^{-1}.$$ Note that $\P(\alpha,\beta;\l;\kappa_\wp)$ is independent of the choice of $\omega_\wp$.  
	\medskip

	The connection between $\P(HD,\cdot)$ and Galois representation $\rho_{HD,\ell}$ is the following result of Katz. 
	
	\begin{theorem}[Katz \cite{Katz, Katz09}]\label{thm:Katz}Let  $\ell$ be a prime. Given a primitive pair of multi-sets $\alpha=\{a_1,\cdots,a_n\}$, $\beta=\{1,b_2,\cdots,b_n\}$ with $M = lcd(\alpha \cup \beta)$, 
		for any datum HD = $\{\alpha, \beta; \l\}$ with $\l \in \Z[\zeta_M,1/M]\smallsetminus \{0\}$ the following hold. 
		\begin{itemize}
			\item [i).]There exists an $\ell$-adic Galois representation $\rho_{HD,\ell}: G(M)\rightarrow GL(W_{\l})$ unramified almost everywhere such that at each 
			prime ideal $\wp$ of  $ \Z[\zeta_M,1/(M\ell \l)]$ with norm $N(\wp)=|\kappa_\wp|$,
			
			\begin{equation}\label{eq:Tr1} \Tr \rho_{HD,\ell}(\text{Frob}_\wp)= (-1)^{n-1}{\omega_{\wp}^{(N(\wp)-1)a_1}(-1)}
			\P(\alpha,\beta; 1/\l;\kappa_\wp),  
			\end{equation} 
			where  $\Frob_\wp$ stands for the  geometric  Frobenius conjugacy class
			of $G(M)$ at $\wp$.

			\item[iia).] When $\l\neq 1$,  the dimension $d :=dim_{\overline \Q_\ell}W_{\l}$ equals $n$ and all roots of the characteristic polynomial of $\rho_{HD,\ell}(\Frob_\wp)$  are algebraic numbers and have the same absolute value $N(\wp)^{(n-1)/2}$ under all archimedean embeddings.
			
			\item[iib).] When $\l\neq 1$ and $HD$ is self-dual  (see Definition \ref{def:1} below),  then $W_{\l}$ admits a symmetric (resp. alternating) bilinear pairing if
			$n$ is odd (resp. even).
			
			\item[iii).] When $\l=1$, dimension  of $W_\lambda$ equals $n-1$. In this case if $HD$ is self-dual, then $\rho_{HD,\ell}$ has a subrepresentation $\rho_{HD, \ell}^{prim}$ of dimension $2\lfloor \frac {n-1}2 \rfloor$ whose representation space admits a symmetric (resp. alternating) bilinear pairing if $n$ is odd (resp. even). All roots of the characteristic polynomial of $\rho_{HD, \ell}^{prim}(\Frob_\wp)$ have absolute value $N(\wp)^{(n-1)/2}$, the same as  (iia). 
		\end{itemize}
	\end{theorem}
	Here and in what follows, when $\ord_\wp \l \ge 0$, the $\l$ in $\P(\alpha, \beta; \l; \kappa_\wp)$ is viewed as an element in the residue field $\kappa_\wp$. The reader is referred to \cite{Katz, Katz09} or  \cite{LTYZ}  for a discussion related to the formulation of Theorem \ref{thm:Katz}.   
	The relation \eqref{eq:Tr1}  sometimes appears as equating certain hypergeometric character sums with either coefficients of modular forms or Jacobi sums,  see \cite[etc.]{Win3a, FOP, Fuselier10,Fuselier13, LTYZ, Ono98, Salerno13, Tu-Yang-evaluation}.

	\begin{remark}\label{phi(M,a_1)}The sign $\omega_{\wp}^{(N(\wp)-1)a_1}(-1)$, as $\wp$ varies, defines a character $\phi(M,a_1)$ of $G(M)$, which is trivial unless $\ord_2 ~M = -\ord_2 ~a_1 = r \ge 1$. In the latter  case it is the quadratic character corresponding to the Hilbert's quadratic norm residue symbol $\(\frac{\zeta_{2^r}}{\cdot}\)_2$ on the field $\Q(\zeta_M)$.
	\end{remark}
	
	\subsection{The purpose of this paper}\label{sec:purpose} 
	A well-known formula of Whipple relates $_7F_6(1)$ to $_4F_3(1)$. When specialized to the self-dual form with a fixed prime $p$, it reads as follows: 
	{ \begin{multline}\label{eq:Whipple2}
		\pFq{7}{6}{\f12&  {\frac 54}&c&1-c&\frac{1-{p}}2&f&1-f}{&{\frac14}& \frac32-c&
			\f12+c&1+\frac {{p}}2&\frac32-f&\f12+f}{1}=\\  \frac{\G(\frac p2)\G(\frac32-f)\G(\f12+f)\G(\frac p2)}{\G(\frac12)\G(\frac{1}2)\G(1+\frac p2-f)\G(\frac p2+f)}\times\left( p\cdot \pFq{4}{3}{\f12&\frac{1-  {p}}2&f &1-f}{&1-\frac {{p}}2&\frac32-c&\f12+c}{1}\right).
		\end{multline}}
	
	The reader is referred to \S \ref{sec:Whipple} for the derivation. The  hypergeometric datum corresponds to the left hand side (after deleting the pair $\frac 54$ and $\frac 14$ that correspond to the same character) is  
	\begin{multline*}
	HD_1(c,f):= \\
	\left \{ \alpha_6(c,f):=\left \{ \frac12,c,1-c,\frac12,f,1-f \right\},   \beta_6(c,f)=\left \{ 1,\frac32-c,\frac12+c,1,\frac32-f,\f12+f\right\};  {1}\right\};
	\end{multline*}
	while that corresponds to the right hand side is 
	\begin{equation*}
	HD_2(c,f):= \left\{\alpha_4(f):=\lc \f12, \bk \f12, f,1-f\rc,\beta_4(c):=\lc  1, \bk 1,\frac32-c,\f12+c\rc;  {1} \right\}.\end{equation*} 
	
	The formula (\ref{eq:Whipple2}) says a period of the hypergeometric datum $HD_1(c,f)$ is ``degenerate" in the sense that it can be  represented by a period of the hypergeometric datum $HD_2(c,f)$.
	The first main result of this paper is to give a Galois representation-theoretic interpretation of the above Whipple formula for special pairs of $(c,f)$.  
	
	\begin{theorem}\label{thm:LLT2main}For $(c,f)$   such that both $HD_1(c,f)$ and $HD_2(c,f)$ are primitive,  let $ M(c,f) := lcd( HD_2(c,f) \bk)$, and $N(c,f) := lcd(\frac{1+2f-2c}4, \frac{3-2f-2c}4)$. Then either $N(c,f)=2 M(c,f)$ or $N(c,f)=M(c,f)$. Given any prime $\ell$, the semi-simplification ${\rho_{HD_1(c,f),\ell}}^{ss}$ of $\rho_{HD_1(c,f),\ell}$ decomposes as 
		$$  {\rho_{HD_1(c,f),\ell}}^{ss}|_{G(N(c,f))}\cong (\epsilon_\ell \otimes   { \rho_{HD_2(c,f),\ell})}|_{G(N(c,f))} \oplus \sigma_{sym,\ell}$$  where $\epsilon_\ell$ is the $\ell$-adic cyclotomic character, and $\sigma_{sym,\ell}$ is a 2-dimensional representation of  $G(N(c,f))$ that can be computed explicitly. 
	\end{theorem} \bk
	
	There are only seven un-ordered  pairs of $(c,f)\in \Q^2$ for which $HD_1(c,f)$ is defined over $\Q$ (see next section for definition) 
	and primitive. They are
	\begin{equation}\label{eq:7pair}
	\left(\f12,\f12\right),\left(\f12,\frac13\right),\left(\f12,\frac16\right),\left(\frac13,\frac13\right),\left(\frac16,\frac16\right),\left(\frac 15,\frac25\right),\left(\frac 1{10},\frac3{10}\right).
	\end{equation} 
	\noindent Among them $HD_2(c,f)$ is primitive for all pairs,  
	but is defined over $\Q$ only for the first five pairs. Direct computations show that $\rho_{HD_1(c,f),\ell}$ and $\rho_{HD_2(c,f),\ell}$ are semi-simple (cf. Theorems \ref{thm:2} and \ref{thm:HD2} in \S\ref{sec:mainresults}). 
	Thus we obtain the following specialization of Theorem \ref{thm:LLT2main}.
	
	\begin{theorem}\label{thm:main} For the seven pairs of $(c,f)$ listed in (\ref{eq:7pair}), $N(c,f) = 2 M(c,f)$ except for $(c,f) = (\frac12, \frac13)$, in which case $N(c,f) = M(c,f)$. Further 
		$$  {\rho_{HD_1(c,f),\ell}}|_{G(N(c,f))}\cong (\epsilon_\ell \otimes   { \rho_{HD_2(c,f),\ell})}|_{G(N(c,f))} \oplus \sigma_{sym,\ell},$$
		where $\sigma_{sym,\ell}$ is a 2-dimensional representation of  $G(N(c,f))$. 
	\end{theorem}

	For the pairs in (\ref{eq:7pair}), the representations $\rho_{HD_i(c,f),\ell}$, $i = 1,2$, can be extended to $G_\Q$. One of the extensions, $\rho_{HD_i(c,f),\ell}^{BCM}$, was studied by Beukers, Cohen and Mellit \cite{BCM}, as explained in \S \ref{sec:BCM}.
	
	Our second purpose is to address, for the seven  cases, three natural questions, explained below, connecting different aspects of the  hypergeometric data, extending the result in \cite{Long18} by Long.
	
	\begin{itemize}
		\item[a)] {\bf Modularity.} A Galois representation  
		is called \emph{modular} or \emph{automorphic} if it corresponds to an automorphic representation of a general linear group as described in the Langlands program. As both $\rho_{HD,\ell}$ and $\rho_{HD,\ell}^{BCM}$ are nice Galois representations arising from geometry and unramified almost everywhere, they are expected to be automorphic. Can this be established when they are  representations of $G_\Q$, such as $\rho_{HD,\ell}^{BCM}$?

		\item[b)]{\bf Truncated HGS and  $p$-adic periods.} Assume that at $p$, the characteristic polynomial of $\rho_{HD,\ell}^{BCM}(\Frob_p)$ has a {\bf unique} unit root $\gamma_p$.  When $\beta=\{1,1,\cdots,1\}$, Dwork showed that,  when $F(\alpha,\beta;\l)_{p-1}\neq 0 \mod p,$ such  a $\gamma_p$ exists and can be computed from truncated HGS  $$\gamma_p:=\gamma_p(HD)\equiv F(\alpha,\beta;\l)_{p^{s}-1}/F(\alpha,\beta;\l^p)_{p^{s-1}-1}\mod p^s.$$ 
		If $\rho_{HD,\ell}^{BCM}$ is modular, is there a $2$-dimensional component corresponding to a Hecke-eigenform $g$ such that at such $p$, the Hecke-eigenvalue $a_p(g)$ of $g$ satisfies the congruence relation $a_p(g) \equiv \gamma_p \mod p$?
		
		\item[c)]{\bf  HGS and Complex periods.} If the answers to  both questions a) and b) are positive, is the untruncated hypergeometric value $F(\alpha, \beta;\l)$ related to a 
		period of the normalized Hecke-eigenform $g$ which affords the unit roots in b)?
	\end{itemize}

	\subsection{An example}\label{sec:example}
	
	We illustrate an example of Katz representation $\rho_{HD,\ell}$ of $G_\Q$, and discuss the three questions a), b), c) raised above.

	Let $ \alpha=\{\f12,\frac12,\frac12\},\beta=\{1,1,1\}$, which form a  primitive pair with $M = lcd(\alpha, \beta) = 2$. For any $\l\in \Q^\times$, $p\nmid lcd(\{\frac12,\l\})$, by definition,
	\begin{eqnarray*} 
		\P(HD; p) = \P(\alpha,\beta;\l; \F_p) &=& \pPPq{3}{2}{\omega_p^{(p-1)/2}&\omega_p^{(p-1)/2}&\omega_p^{(p-1)/2}}{&\eps&\eps}{\l;p}\\
		&=&\sum_{x, y \in \F_p} \phi_p(xy(y-1)(1-x)(x-\l y)),
	\end{eqnarray*} where $\eps$ is the trivial character and $\phi_p = \omega_p^{(p-1)/2}$ is the quadratic character of $\F_p^\times$ extended to $\F_p$ by $\phi_p(0) = 0$. 
	
	Now consider the following surface defined by the  affine  equation 
	\begin{equation}\label{eq:A}
	\mathcal A_\l: \quad s^2=xy(x-\l y)(y-1)(1-x).
	\end{equation}It is a double cover of $\C P^2$ ramified at 6 lines. Its smooth model is a K3 surface obtained by blowing up the intersection points as described in \cite[\S 8]{Stienstra-Beukers} by Stienstra and Beukers.  The smooth model has good reduction at  $p$. The  number of affine $\F_p$-rational points of $ \mathcal A_\l$ is equal \bk to 
	$$ \sum_{x,y\in \F_p}(1+\phi_p(xy(x-\l y)(y-1)(1-x)))=p^2+\P(HD;p).$$ For a generic $\l$, the Picard number of the K3 surface $\mathcal A_\l$ is 19, as shown in \cite{AOP} by Ahlgren, Ono, and Penniston or \cite{Long03} by Long. The associated representation $\rho_{HD,\ell}$ of $G_\Q$ is 3-dimensional which is the symmetric square of a 2-dimensional Galois representation arising from an explicit elliptic curve $E_\l$, see \cite[(2)]{AOP}. This essentially  follows from the Clausen formula, see (\ref{eq:Clausen}) and  Theorem \ref{thm:E-G} for the finite field version by Evans and Greene. When $\l=1$, $\mathcal A_1$ has Picard number 20 and $\rho_{HD,\ell}$ is 2-dimensional. 
	As shown in \cite[section 5]{LLL} by the first two authors and Liu, the action of the Galois group $G_\Q$  on the  $2$-dimensional subspace of the cohomological space $H_{\acute{e}t}^2(\mathcal A_1\otimes_\Q \overline{\Q}, \Q_\ell)$ spanned by the transcendental cycles on $\mathcal A_1$ has the same trace as $\rho_{HD,\ell}$ twisted by $\phi(2, 1/2)= \(\frac{-1}{\cdot}\){=\chi_{-1}}$. In this paper for a square-free integer $d$,  $\chi_d$  denotes the quadratic character of $G_\Q$ with kernel $G_{\Q(\sqrt d)}$; it is also identified with the quadratic character of $\Q$ 
	associated to the extension $\Q(\sqrt d)/\Q$. 
	On the other hand, since $\rho_{HD,\ell}(\Frob_p)$ has trace zero at $p \equiv 3 \mod 4$ where $\(\frac{-1}{p}\) = -1$, $\rho_{HD,\ell}$ is invariant under twist by $\phi(2, 1/2)$ introduced in Remark \ref{phi(M,a_1)} and we get the representation  $\rho_{HD,\ell}$.

	It is explained in \cite{Stienstra-Beukers} by Stienstra and Beukers (see also \cite[section 3]{ALL} and \cite[section 5]{LLL}) that $\mathcal A_1$ is birationally isomorphic to an elliptic surface fibered over a genus $0$ curve $X$ over $\Q$. 
	The above $G_\Q$ action on the $2$-dimensional subspace of $H_{\acute{e}t}^2(\mathcal A_1\otimes_\Q \overline{\Q}, \Q_\ell)$ can also be realized on $H^1$ of $X \otimes_\Q \overline{\Q}$ with coefficients in an $\ell$-adic sheaf on $X$. This is the prototype of the Katz representation $\rho_{HD,\ell}$ in Theorem \ref{thm:Katz} for a general hypergeometric datum HD.  
	\bigskip

	Ahlgren in \cite{Ahlgren01} showed that for any prime $p>2$, $$\P(HD;p)=\P\left( \left\{\f12,\frac12,\frac12\right\},\{1,1,1\};1; \F_p\right)=a_p(\eta(4\tau)^6),$$the $p$th coefficient of $\eta(4\tau)^6$, which admits complex multiplication (CM) by {$\chi_{-1}$}. Here $\eta(\tau)$ is the Dedekind eta function.
	In other words, $ \rho_{HD,\ell}$   is isomorphic to the Galois representation attached to $\eta(4\tau)^6$, hence it is modular. 
	This is a). For b), Dwork's theorem holds for the ordinary prime $p\equiv 1 \mod 4$. In particular, the unit root can be written as $\gamma_p=-\G_p(\frac14)^4$  where $\G_p(\cdot)$ denotes the $p$-adic Gamma function.
	In fact, Ahlgren showed that a stronger congruence, called supercongruence, holds. Namely for each prime $p> 2$,
	\begin{equation} F\left (\left\{ \f12,\frac12,\frac12\right \},\{1,1,1\};1 \right)_{p-1}\equiv a_p(\eta(4\tau)^6) \mod p^2.\end{equation}In \cite{LR} Long and Ramakrishna further showed that when $p>2$, \begin{equation} F\left (\left\{ \f12,\frac12,\frac12\right \},\{1,1,1\};1 \right)_{p-1}\equiv \begin{cases} -\G_p(\frac14)^4 \quad \text{ if } p\equiv 1\mod 4\\ -\frac{p^2}{16}\G_p(\frac14)^4 \quad \text{ if } p\equiv 3\mod 4 \end{cases}  \mod p^3.\end{equation}
	An archimedean version for the untruncated hypergeometric value is proved in  our paper  \cite{LLT}:
	\begin{equation}\pFq32{\f12,\f12,\frac12}{&1,1}1= \frac{16}{\pi^2}L(\eta(4\tau)^6,2)=\frac{\G(\frac14)^4}{4\pi^3}, 
	\end{equation}which is c). Similar recent discussions relating $L$-values of hypergeometric motives to $L$-functions of modular forms or Hilbert modular forms include \cite{Osburn-Straub} by Osbrun, Straub, and \cite{DPVZ} by Demb\'el\'e, Panchishkin, Voight, and Zudilin.

	\subsection{Main results on modularity, congruences and periods}\label{sec:mainresults}

	To each of the seven pairs $(c,f)$ in (\ref{eq:7pair}), 
	by Theorem \ref{thm:Katz} there are associated representations $\rho_{HD_1(c,f), \ell}$ and $\rho_{HD_2(c,f),\ell}$ of the Galois group $G(M(c,f))$,    
	where $M(c,f) = lcd(\alpha_6(c,f)\cup \beta_6(c,f)) = lcd(\alpha_4(f), \beta_4(c))$, of degrees $5$ and $3$, respectively. Moreover, $\rho_{HD_1(c,f), \ell}$ can be extended to $\rho_{HD_1(c,f), \ell}^{BCM}$ of $G_\Q$ by Theorem \ref{thm:KBCM}. The same holds for $\rho_{HD_2(c,f), \ell}$ for the first five pairs; and as shown in \S \ref{HD2},  this is also true for the remaining two pairs.

	It follows from  Theorem \ref{thm:KBCM}  that, for the seven pairs (resp. the first five pairs) of $(c,f)$ in \eqref{eq:7pair}, $\rho_{HD_1(c,f),\ell}^{BCM}$ (resp. $\rho_{HD_2(c,f),\ell}^{BCM}$) has a $4$-dimensional (resp. $2$-dimensional) sub-representation $\rho_{HD_1(c,f),\ell}^{BCM,prim}$ (resp. $\rho_{HD_2(c,f),\ell}^{BCM,prim}$) and a $1$-dimensional complement $\rho_{HD_1(c,f),\ell}^{BCM,1}$ (resp. $\rho_{HD_2(c,f),\ell}^{BCM,1}$).

	\medskip
	
	We begin with the modularity question a). First consider $\rho_{HD_2(c,f), \ell}^{BCM}$.
	
	\begin{theorem}\label{thm:HD2}
		For each of the first five pairs of $(c,f)$ such that $HD_2(c,f)$ is defined over $\Q$, both components of $\rho_{HD_2(c,f),\ell}^{BCM}$ are modular, which are explicitly identified by the  information at primes $p>7$ in the table below. For each of the remaining two pairs, $M(c,f) = 10$ and the representation $\rho_{HD_2(c,f),\ell}$ of $G(10)$ can be extended to $G_\Q$, with both components modular. For each pair we list extensions with integral traces under ``$\rho_{HD_2(c,f),\ell}^{BCM,prim}$" and ``$\rho_{HD_2(c,f),\ell}^{BCM,1}$", resp.   
		$$ 
		\begin{array}{|c|c|c|c|c|c|c|c|}
		\hline
		(c,f)&\alpha&\beta&w(HD_2)&\Tr\, \rho_{HD_2(c,f),\ell}^{BCM,prim}(\Frob_p)& \rho_{HD_2(c,f),\ell}^{BCM,1}\\
		\hline
		(\f12,\frac12)&(\f12,\f12,\f12,\f12)&(1,1,1,1)&4& a_p(f_{8.4.a.a})& \epsilon_\ell \\ \hline
		(\f12,\frac13)&(\f12,\f12,\frac13,\frac23)&(1,1,1,1)& 4& a_p(f_{36.4.a.a})=\left (\frac{-3}p \) a_p(f_{12.4.a.a})& \chi_{3} \cdot \epsilon_\ell\\
		\hline
		(\frac13,\frac13)&(\f12,\f12,\frac13, \frac23)&(1,1,\frac76,\frac56)&  4&  a_p(f_{6.4.a.a})= \left (\frac{-3}p \) a_p(f_{18.4.a.a})& \chi_{-3}\cdot\epsilon_\ell \\
		\hline
		(\f12,\frac16)&(\f12,\f12,\frac16, \frac56)&(1,1,1,1)& 4& a_p(f_{72.4.a.b})=\left (\frac{-3}p \) a_p(f_{24.4.a.a})&\epsilon_\ell\\
		\hline
		(\frac16,\frac16)&(\f12,\f12,\frac16,\frac56)&(1,1,\frac43,\frac23)& 2& a_p(f_{24.2.a.a})\cdot p={\left (\frac{-3}p \)} a_p(f_{72.2.a.a})\cdot p& \chi_{-3}\cdot\epsilon_\ell  \\\hline
		(\frac15,\frac25)&\ (\f12,\f12,\frac{2}5,\frac35)&(1,1,\frac{13}{10},\frac7{10})& 4& a_p(f_{50.4.a.b})=\(\frac5p \) a_p(f_{50.4.a.d})&\epsilon_\ell  \text{ or } \chi_{5}\cdot \epsilon_\ell   \\\hline
		(\frac1{10},\frac3{10})& (\f12,\f12,\frac3{10},\frac7{10})& (1,1,\frac75,\frac35)& 2& a_p(f_{200.2.a.b})\cdot p =\(\frac5p \) a_p(f_{200.2.a.d})\cdot p&  \epsilon_\ell  \text{ or } \chi_{5}\cdot \epsilon_\ell \\
		\hline
		\end{array}
		$$
	\end{theorem} 
	
	The modular forms above are denoted by their LMFDB labels; for instance $f_{8.4.a.a}$ means a normalized cuspidal newform of  level 8 and weight 4. See \eqref{eq:w}  for the definition of $w(HD_2)$.  
	\medskip

	Our next theorem  records the modularity  of $\rho_{HD_1(c,f),\ell}^{BCM}$.
	
	\begin{theorem}\label{thm:2}
		For each pair $(c,f)$ in \eqref{eq:7pair}, $\rho_{HD_1(c,f),\ell}^{BCM,prim}$ decomposes as a sum of two $2$-dimensional $G_\Q$-modules, both of these two subrepresentations are modular. There are two normalized cuspidal newforms $\mathfrak {f}_{c,f}$ and $\mathfrak {g}_{c,f}$ (expressed in their LMFDB labels) such that for each prime $p\ge 7$, $\Tr \,\rho_{HD_1(c,f),\ell}^{BCM}(\Frob_p)$ is explicitly described by the Fourier coefficients $a_p(\mathfrak {f}_{c,f})$ and $a_p(\mathfrak {g}_{c,f})$ in the following table, where $\ff_{c,f}$ (resp. $\fg_{c,f}$ ) is the first (resp. second) modular form listed in the last column.

		$$
		\begin{array}{|c|c|c|c|c|c|c|c|}
		\hline
		(c,f) & \mathcal J(HD_1(c,f); \F_p)&w(HD_1)&w(HD_2)&\Tr\, \rho_{HD_1(c,f),\ell}^{BCM}(\Frob_p)   \\ \hline
		(\f12,\frac12)  &-\(\frac{-1}p\)   &6 &4&a_p(f_{8.6.a.a}) +p\cdot a_p(f_{8.4.a.a})+\(\frac {-1}p\)p^2\\ \hline
		\(\frac 12,\frac13\)&-\(\frac{-1}p\)p  &6 &4&a_p(f_{4.6.a.a}) +p\cdot a_p(f_{12.4.a.a})+\(\frac 3p\)p^2\\
		\hline
		(\frac13,\frac13) &-\(\frac{-1}p\)p^2  &6&4&a_p(f_{6.6.a.a})+ p\cdot a_p(f_{18.4.a.a})+\(\frac{-1}p\)p^2 \\ \hline
		(\f12,\frac16)&-\(\frac{-1}p\)p &4&4&p\cdot a_p(f_{8.4.a.a})   + p\cdot a_p(f_{24.4.a.a})+ \(\frac 3p\)p^2 \\
		\hline  (\frac16,\frac16) &-\(\frac{-1}p\)p^2   &2&2&p^2\cdot a_p(f_{24.2.a.a})+p^2\cdot a_p(f_{72.2.a.a})+\(\frac{-1}p\)p^2 \\  \hline
		(\frac15,\frac25)     &-\(\frac{-1}p\)p^2 &4&4&p\cdot a_p(f_{10.4.a.a})+p\cdot a_p(f_{50.4.a.d})+\(\frac{-5}p\)p^2\\
		\hline
		(\frac1{10},\frac3{10})& - \(\frac{-1}p\)p^2  &2&2&p^2\cdot a_p(f_{40.2.a.a})+p^2\cdot a_p(f_{200.2.a.b})+\(\frac {-5}p\)p^2\\ \hline
		\end{array}$$In the above Table,  the term  $\mathcal{J}(HD_1(c,f);\F_p)$ is defined by \eqref{eq:J}, and  the weights  $w(HD_1)$  and  $w(HD_2)$ are defined in \eqref{eq:w}.
	\end{theorem}

	As shown in (\ref{eq:BCMtrace}), $\Tr\, \rho_{HD_i(c,f),\ell}^{BCM}(\Frob_p)$ is expressed in terms of the finite field hypergeometric function $H_p$ introduced in \cite{BCM}, see \eqref{eq:hq} for its definition by Gauss sums. Its relationship to the $\P$-function is given in (\ref{eq:P-H}). Hence the above two theorems can be reformulated in terms of character sums. For example, 
	the case $(\f12,\frac13)$ is equivalent to, for $p>5$, 
	$$
	H_p\(HD_2(\f12,\frac13)\)= a_p(f_{36.4.a.a})+\(\frac 3p\)p 
	$$
	and 
	$$
	pH_p\(HD_1(\f12,\frac13)\)= a_p(f_{4.6.a.a})+pa_p(f_{36.4.a.a})+\(\frac 3p\)p^2.
	$$

	When $(c,f)=(\f12,\f12)$, the conclusion above is equivalent to a conjecture of Koike \cite{Koike92}, proved in \cite{FOP} by Frechette, Ono and Papanikolas. 
	Comparing Theorems \ref{thm:HD2} and \ref{thm:2}, we see that the extension  to $G_\Q$ of $\epsilon_\ell|_{G(M(c,f))} \otimes \rho_{HD_2(c,f),\ell}^{prim}$ in $\rho_{HD_1(c,f),\ell}^{BCM}$ differs from $\epsilon_\ell \otimes \rho_{HD_2(c,f),\ell}^{BCM,prim}$ by a character of $G_\Q/G(N(c,f))$. The same phenomenon occurs for the $1$-dimensional component. 
	\medskip
	
	For the question b), we see from the above table that only when $(c,f)=(\f12,\frac12), (\f12,\frac13),(\frac13,\frac13)$,  $\rho_{HD_1(c,f),\ell}^{BCM}$ could possibly have a unique unit root $\gamma_p$ at an  unramified prime $p$. 
	In each of these cases, there is a congruence relation modulo $p$ between a truncated hypergeometric series and the coefficients of $\ff_{c,f}$, see Theorem 4 of \cite{Long18}. Numerical computations suggest the following supercongruences for  primes $p>5$: 
	
	\begin{eqnarray*}
		(c,f) =(\f12,\f12): & \pFq65{\f12&\f12&\f12&\f12&\f12&\f12}{&1&1&1&1&1}{1}_{p-1}&\overset{?}\equiv a_p(f_{8.6.a.a})\mod p^5.\\
		(c,f)=(\f12,\frac13): &  p\cdot \pFq 65{\f12&\f12&\f12&\f12&\frac13&\frac23}{&\frac56&\frac76&1&1&1}{1}_{p-1}&\overset{?}\equiv a_p(f_{4.6.a.a}) \mod p^5.\\
		(c,f)=(\frac13,\frac13): &p^2\cdot \pFq 65{\f12&\frac13&\frac23&\f12&\frac13&\frac23}{&\frac56&\frac76&1&\frac56&\frac76}{1}_{p-1}&\overset{?}\equiv a_p(f_{6.6.a.a}) \mod p^4.
	\end{eqnarray*}  Note that these supercongruences hold for both ordinary and supersingular primes. 
	
	In the above formulas, each $_6F_5$ is multiplied by a suitable $p$ power to ensure the resulting product is $p$-adically integral. Of these, the first congruence is shown to hold mod $p^3$ by Osburn, Straub and Zudilin in \cite{OSZ}. It would be interesting to know the truth. The reader is referred to \cite{Kilbourn, Long18,LTYZ, McCarthy-RV,RR} for related discussions on hypergeometric supercongruences.

	Finally we provide an archimedean version of these congruences for $(c,f)=(\f12,\frac12)$ and $(\f12,\frac13)$. Namely we give precise descriptions of the untruncated hypergeometric values as periods of $\ff_{c,f}$, which addresses the question c).  Interestingly, for $(c,f)=(\f12,\frac16)$,  we find an expression of 
	$F(HD_1(c,f))$ in terms of a period of $f_{8.4.a.a}$, see \eqref{eq:7F6-f12-2}.

	\medskip
	
	Similar to the identification of $\ff_{c,f}$ at the archimedean place, we will also explain in the $(\f12,\f12)$ case the natural appearance of $\fg_{c,f}$ from the perspective of modular forms. 
	Meanwhile, numerical data suggest for any  prime $p>3$
	\begin{equation}\label{eq:7F(1)-1} 
	\pFq76{\f12&\frac54&\f12&\f12&\f12&\f12&\f12}{&\frac14&1&1&1&1&1}{1}_{p-1}\overset{?}\equiv p\cdot a_p(f_{8.4.a.a})\mod p^4.
	\end{equation}The next result  expresses the untruncated hypergeometric $_6F_5(1)$ and $_7F_6(1)$ values in terms of periods of modular forms $\ff_{\f12,\f12}$ and $\fg_{\f12,\f12}$ respectively. For a general frame work for periods in terms of integrals, see  \cite{K-Z} by Kontsevich and Zagier, in which the authors  pointed out that ``Periods appear surprisingly
	often in various formulas and  conjectures in mathematics, and often provide a bridge
	between problems  coming from different disciplines."

	\begin{theorem}\label{thm:period}The following two equalities hold:
		\begin{equation*} 
		\pFq65{\f12&\f12&\f12&\f12&\f12&\f12}{&1&1&1&1&1}{1} =16\oint_{|t_2(\tau)|=1} \tau^2 f_{8.6.a.a}\left (\frac{\tau}2\right)d\tau=16 \int_{1/2+i/2}^{-1/2+i/2} \tau^2 f_{8.6.a.a}\left (\frac{\tau}2\right) d\tau
		\end{equation*}
		and
		\begin{equation}\label{eq:7F6-f12-2}  
		\pFq76{\f12&\frac54&\f12&\f12&\f12&\f12&\f12}{&\frac14&1&1&1&1&1}{1}=\frac{32i}{\pi} \oint_{|t_2(\tau)|=1}\tau f_{8.4.a.a}(\frac{\tau}2)d\tau =\frac{32i}{\pi}  \int_{1/2+i/2}^{-1/2+i/2} \tau f_{8.4.a.a}(\frac{\tau}2) d\tau,
		\end{equation}where $ t_2(\tau)=-64\frac{\eta(2\tau)^{24}}{\eta(\tau)^{24}}$ is a Hauptmodul for $\Gamma_0(2)$, 
		and the path in $\tau$ such that $|t_2(\tau)|=1$ is the hyperbolic geodesic from $\frac{1+i}2$ to $\frac{-1+i}2$ going counter clockwise. 
	\end{theorem}

	We believe the appearance of $\fg_{c,f}$ in the $(\f12,\frac13)$ case can be explained in a similar way. This is left to any interested reader.

	\subsection{Geometric perspective}
	Given $k\in \C\smallsetminus \{0,\pm 1\}$, the 1st de Rham cohomology $H^1_{DR}(L(k)/\C)$ of the elliptic curve 
	$L(k):\quad y^2=(1-x^2)(1-k^2 x^2)$ is isomorphic to the differentials on $L(k)$ with at most a double pole at infinity (see \cite[Appendix 1]{Katz-padic} by Katz). It is a 2-dimensional vector space over $\C$ generated by  
	$\mathfrak w_1(k)=\frac {dx}{y}$ and $\mathfrak w_2(k)=(1-k^2x^2)\frac {dx}{y}$. Here $\mathfrak w_1(k)$ is holomorphic, which is unique up to scalar. Typically there is no a priori `functorial' choice for the second generator of $H_{DR}^1(L(k)/\C)$ unless $L(k)$ admits complex multiplication (CM), see \cite{Katz-padic}. 
	The two periods $$K(k)=\int_0^1 \mathfrak w_1(k)= \frac{\pi}2 \pFq21{\f12,\f12}{,1}{k^2}$$ and  $$E(k)=\int_0^1 \mathfrak w_2(k)=\frac{\pi}2 \pFq21{-\f12,\f12}{,1}{k^2}$$ are complete elliptic integrals of first and second kind respectively, see \cite [Chapter 1]{BB} by Borwein-Borwein. As $k$ varies, these two period functions are related by
	$$E(k) =k(1-k^2)\frac{d}{dk}K(k)+(1-k^2)K(k).$$
	
	In our situation, the hypergeometric functions $$F(\l):=F(\alpha_6(c,f),\beta_6(c,f);\l)$$  and  $$_7F_6(\l):=\pFq{7}{6}{\f12&\frac 54&c&1-c&\frac{1}2&f&1-f}{&\frac14& \frac32-c&
		\f12+c&1&\frac32-f&\f12+f}{\l}=4\frac{d }{d\l}F(\l)+F(\l)$$ play roles parallel to $K(k)$ and $E(k)$, respectively. Geometrically, for any given primitive pair $\alpha,\beta$ of length $n$ and $\l\in \overline {\Q}^\times$, one can write down (using the inductive formula  via Euler integral (see \cite[\S 2.1]{Slater}) for $_nF_{n-1}$) an affine equation, which defines a variety $V(\l):=V_{\alpha,\beta}(\l)$. The surface $\mathcal A_\l$ in \S\ref{sec:example}  is such an example. The variation of the de Rham cohomology  $H_{DR}^{n-1}(V(\l)/\C)$ can be described by the hypergeometric differential equation satisfied by $F(\alpha,\beta;\l)$. By the Hodge filtration, we have  $H_{DR}^{n-1}(V(\l))\cong \bigoplus_{i=0}^{n-1} H^i(V(\l),\Omega_{V(\l)}^{n-1-i})$. For example,  $H^0(V(\l),\Omega_{V(\l)}^{n-1})$ consists of holomorphic differential $n-1$st forms on $V(\l)$. 
	For $HD_1(c,f)=\{\alpha_6(c,f),\beta_6(c,f); 1\}$ with $(c,f)$ in \eqref{eq:7pair},  $\dim H^0(V_{\alpha_6(c,f),\beta_6(c,f)}(1),\Omega^5)=1$ only occurs  for the first three cases of Theorem \ref{thm:2}.

	From this perspective, Whipple's formula \eqref{eq:Whipple2} suggests that when $\l=1$, up to a Tate twist, there is a correspondence between $H_{DR}^5(V_{\alpha_6(c,f),\beta_6(c,f)}(1))$ and  $H_{DR}^3(V_{\alpha_4(c,f),\beta_4(c,f)}(1))$.  Firstly, such a correspondence only exists for special fibers like $\l=1$, similar to $H_{DR}^1(L(k))$ admitting a canonical splitting  only when $L(k)$ has CM. Secondly, such a correspondence, when it exists, can be recognized simultaneously from different aspects in compatible ways, including the statement of Whipple's formula \eqref{eq:Whipple2}, the statements of Theorems \ref{thm:main} and \ref{thm:2}, the conjectural supercongruence \eqref{eq:7F(1)-1}, as well as \eqref{eq:7F6-f12-2}. In particular, the factor $\frac1{\pi}$ on the right hand side of \eqref{eq:7F6-f12-2} and the $p$ on the right hand side of \eqref{eq:7F(1)-1} both reflect a Tate twist by 1. \footnote{The Legendre relation for $L(k)$ is $E(k)K(\sqrt{1-k^2})+K(k)E(\sqrt{1-k^2})-K(k)K(\sqrt{1-k^2})=\frac\pi 2$, which is a  pairing between $H^0(L(k),\Omega_{V(1)})$ and $H^1(L(k), \mathcal O_{V(1)})$.}
	
	Finally we comment on the role of Theorem \ref{thm:period}. From transcendence theory, given a normalized weight-$k$ cusp form $\ff(\tau)$, the values $\oint_{\gamma}\ff(\tau)p(\tau) d\tau$ for all polynomials   $p(\tau)\in \overline \Q[\tau]$ of degree at most $k-1$ and all closed paths $\gamma$ on the underlying modular curve generate a finite dimensional vector space over $\overline \Q(\pi)$. Thus the equalities in Theorem \ref{thm:period} are not coincidental by nature. On the other hand, from the perspective of de Rham cohomology, one expects that untruncated hypergeometric values arise from line integrals of some differentials in $H_{DR}^{n-1}(V(\l))$. 
	The challenge becomes  how to realize hypergeometric values as periods.
	
	\medskip
	
	This paper is organized in the following way. In \S \ref{ss:Preliminaries}, we recall basic notation, relevant results for hypergeometric data, character sums, liftings of Katz representations, 
	and modular forms. Section \ref{ss:Galois} is devoted to the proofs of Theorems \ref{thm:LLT2main}, \ref{thm:main}, \ref{thm:HD2} and \ref{thm:2}. In \S \ref{sec:mf}, Theorem \ref{thm:period} and other computations for untruncated hypergeometric functions are obtained using properties of modular forms and hypergeometric functions.  It is worth pointing out that in the proofs of Theorems \ref{thm:LLT2main} and \ref{thm:period}, the classical Clausen formula (\ref{eq:gen-Clausen}) and its finite field analogue (Theorem \ref{thm:E-G}) play a prominent role. The finite field version of the well-posedness of the $_6F_5(HD_1)$ under consideration gives rise to Lemma \ref{lem:6P5(1)}, using which we decompose the representation $\rho_{HD_1,\ell}$.

	\section{Preliminaries}\label{ss:Preliminaries}
	
	\subsection{Hypergeometric functions over finite fields}\label{ss:HG-FF}

	We recall the period functions for finite fields with odd characteristic. All multiplicative characters $A$ are extended to $0$ by letting $A(0) = 0$. In what follows,  denote by $\eps$ the trivial character, by $\phi_p$ or simply $\phi$ the quadratic character, and by $\overline A$ the inverse of the character $A$. For any characters  $A_i$, $B_i$, $i=1,\cdots,n+1$ with $B_1=\eps$ in $\widehat{\F_q^\times}$, the $_{n+1}\P_n$-function is defined as follows: set
	$$
	\pPPq 10{A_1}{\,}{\l;q}:=\overline {A}_1(1-\l)
	$$
	and construct $_{n+1}\P_n$ inductively by
	\begin{equation}\label{eq:P-by-induction}
	\pPPq{n+1}n{A_1& A_2&\cdots &A_{n+1}}{&  B_2&\cdots &B_{n+1}}{\lambda;q}
	: =   \sum_{x\in\F_q} A_{n+1}(x)\overline A_{n+1} B_{n+1}(1-x) \cdot
	\pPPq{n}{n-1}{A_1& A_2&\cdots &A_{n}}{&  B_2&\cdots &B_{n}}{\lambda x;q}.
	\end{equation}
	For simplicity, we may leave out the $q$ in the notation when there is no danger of 
	ambiguity in the underlying finite field $\F_q$.
	
	This $\P$-function is defined as an analog of the integral expression of the classical hypergeometric series. See \cite{Greene, Win3X}. 
	Note that for any character $A$ in $\widehat{\F_q^\times}$,   we have
	$$
	\overline {A}(1-\l)=\delta(\l)+\frac1{q-1}\sum_{\chi\in \widehat{\F_q^\times}}J( A\chi,\ol \chi)\chi(-\l)=\delta(\l)+\frac{-1}{q-1}\sum_{\chi\in \widehat{\F_q^\times}}\binom{A\chi}{\chi}\chi(\l),
	$$
	where
	$$  \delta(t):=
	\begin{cases}
	1,& \mbox{ if } t=0, \\
	0,& \mbox{ otherwise, } \end{cases} \quad  \quad J(A, B):=\sum_{t\in \F_q}A(t)B(1-t), \quad \text{and } \displaystyle\CC AB :=-B(-1)J(A,\ol B).
	$$
	Thus, we can rewrite the $\P$-function in terms of the finite field version of binomial coefficients $\binom{A}{B}$ defined above (see \cite[Theorem 3.6]{Greene} by Greene for example):
	\begin{multline}\label{eq:PP}
	\pPPq{n+1}n{A_1& A_2&\cdots &A_{n+1}}{&  B_2&\cdots &B_{n+1}}{\lambda;q}\\
	=\frac{(-1)^{n+1}}{q-1}\cdot \left(\prod_{i=2}^{n+1} A_iB_i(-1) \right)\cdot
	\sum_{\chi\in \widehat{\F_q^\times}}\CC{A_1\chi}{\chi} \CC{A_2\chi}{B_2\chi}\cdots \CC{A_{n+1}\chi}{B_{n+1}\chi}\chi(\l)
	+ \delta(\lambda)\prod_{i=2}^{n+1} J(A_{i},\overline{A_{i}}B_i).
	\end{multline}

	Here are some properties of these character sums that are useful for our later discussion. 
	\begin{prop}[Theorem 4.2, \cite{Greene}]\label{prop: Kummer}
		For any fixed  $\F_q$ of characteristic $p>2$,    $A_i$, $B_j\in \widehat{\F_q^\times}$, and    $t\neq 0$,
		\begin{multline*}
		\pPPq{n}{n-1}{A_1&A_2&\cdots& A_{n}}{&B_2&\cdots&B_n}{\frac1 t}=A_1(-t)\left(\prod_{i=2}^nA_iB_i(-1)\right)\cdot\pPPq{n}{n-1}{A_1&A_1\overline{B_2}&\cdots& A_1\overline{B_n}}{&A_1\overline{A_2}&\cdots&A_1\overline{A_n}}{t}\\
		= A_2(t)\left(\prod_{i=3}^nA_iB_i(-1)\right)\cdot\pPPq{n}{n-1}{A_2\ol B_2&A_2& A_2\overline{B_3}&\cdots& A_2\overline{B_n}}{&A_2\overline{A_1}&A_2\ol{A_3}&\cdots& A_2 \overline{A_n}}{t}.
		\end{multline*}
	\end{prop}

	This proposition is comparable to the following classical result. 
	\begin{prop}
		Given $\alpha=\{a_1,a_2,\cdots, a_n\}$, $\beta=\{1,b_2, b_3, \cdots, b_n\}$, if we  denote $1+a_j-\beta:= \{a_j,1+a_j-b_2,\cdots,1+a_j-b_n\}$ and $1+a_j-\alpha:= \{1+a_j-a_1,\cdots,1+a_j-a_n\}$ then
		the functions
		$F(\alpha,\beta; z )$ and $(-z)^{-a_j}F(1+a_j-\beta,1+a_j-\alpha;1/z)$ for any $j=1,\ldots, n$, satisfy the same differential equation
		$$
		\left[\theta\(\theta+b_2-1\)\cdots \(\theta+b_n-1\) - z\(\theta+a_1\)\cdots \(\theta+a_n\) \right]F=0, \quad \text{where } \theta=z\frac{d}{dz}.
		$$
	\end{prop}
	
	
	Next we recall a version of the classical Clausen formula (\cite[pp 184]{AAR}): 
	
	\begin{equation}\label{eq:gen-Clausen}
	\pFq21{a&b}{&a+b+\f12}{x}\pFq21{\f12-a&\f12-b}{&\frac32-a-b}{x}=\pFq32{\f12&a-b+\f12&b-a+\f12}{&a+b+\f12&\frac32-a-b}{x}.
	\end{equation}
	Below is the Clausen formula  over finite fields obtained  by Evans and Greene. 
	
	\begin{theorem}[Evans-Greene \cite{Evans-Greene}]\label{thm:E-G} 
		Assume $\eta,K\in \widehat{\F_q^\times}$ such that none of $\eta, K\phi,\eta K, \eta \ol K$ is trivial.  Suppose $\eta K = S^2$ is a square. When $t\neq 0,1$, we have 
		\begin{align*}
		\pPPq32{\phi&\eta&\ol \eta}{&K&\ol K}{t}&= \phi(1-t)\left (\phi(t-1)K(t) \frac{J(\phi S, \phi K \ol S)}{J(S, K\ol S)}  \pPPq21{\phi K \ol S& S}{&K}t^2 -  q \right) \\
		&=  \phi(1-t)\left ( \pPPq21{\phi K \ol S& S}{&K}t  \pPPq21{\phi \ol K S & \ol S}{& \ol K}t -  q \right),
		\end{align*} where
		\begin{equation}\label{eq:Grossen-eta}
		\pPPq21{\phi K \ol S& S}{&K}t=\ol K(t)\phi(t-1)\frac{J(S, K\ol S)}{J(\phi K\ol S, \phi S)}   \pPPq21{\phi \ol K S& \ol S}{&\ol K}t.
		\end{equation}
		When $t=1$, we have
		\begin{multline}\label{eq:3P2(1)}
		\pPPq32{\phi&\eta&\ol \eta}{&K&\ol K}{1} 
		=\phi\eta(-1)q\frac{J(\eta K,\ol \eta K)J(\phi,\phi K)}{J(\phi S, \ol K)J(S, \ol K)}\(\frac{J(S,\phi \ol S)}{J(S\ol\eta, \phi\ol S\eta)}+\frac{J(\phi S,\ol S)}{J(\eta\ol S,\phi S\ol\eta)}\)\\
		=\frac{J(\eta K,\ol \eta K)}{J(\phi, \ol K)}\(J(S \ol K, \phi\ol S)^2 + J(\phi S\ol K, \ol S)^2\).
		\end{multline}
		
		In particular, when $K=\eps$, this gives
		\begin{equation}\label{eq:Clausen-K=1}
		\pPPq32{\phi&\eta&\ol \eta}{&\eps&\eps}{1}= J(S,\phi\ol S)^2+ J(\ol S,\phi S)^2.
		\end{equation}   
		When $\eta K$ is not a square, we have 
		\begin{equation}\label{eq:not-square}
		\pPPq32{\phi&\eta&\ol \eta}{&K&\ol K}{1} =0.
		\end{equation} 
	\end{theorem}  The formulation given in \cite{Evans-Greene} is different from the above self-dual form. We will explain in the proof below how to bridge both versions. 
	\begin{proof}
		When $t\neq 0$,
		\begin{multline*}
		\pPPq32{\phi&\eta&\ol \eta}{&K&\ol K}{t}= \frac{-1}{q-1}\sum_{\chi}\CC{\phi \chi}{\chi}\CC{\eta\chi}{K\chi}\CC{\ol\eta\chi}{\ol K\chi}\chi(t)\\
		\overset{\chi\mapsto K\chi}{=} K(t)\frac{-1}{q-1}\sum_{\chi}\CC{\phi K \chi}{K\chi}\CC{\eta K\chi}{K^2\chi}\CC{\ol\eta K\chi}{\chi}\chi(t)
		= \phi(-1)K(t) \pPPq32{\ol\eta K&\eta K&\phi K}{&K^2&K}t.
		\end{multline*}
		Apply the  Clausen's formula over $\F_q$ in \cite[Theorem 1.5]{Evans-Greene} with $C=K$ and $S^2=\eta K$, when $t\neq 1$, one has 
		\begin{multline*}
		\pPPq32{\ol\eta K&\eta K&\phi K}{&K^2&K}t\\=   \frac{J(\phi S, \phi K \ol S)}{J(S, K\ol S)} \left( \pPPq21{\phi K \ol S& S}{&K}t^2 -\phi(1-t)\ol K(t) \cdot \phi(-1)q\cdot \frac{J(S, K\ol S)}{J(\phi S, \phi K \ol S)}\right).
		\end{multline*}
		On the other hand, the claim \eqref{eq:Grossen-eta} follows from \cite[Proposition 8.7 (2)]{Win3X}.
		Therefore, when $t\neq 0$, $1$,
		\begin{align*}
		\pPPq32{\phi&\eta&\ol \eta}{&K&\ol K}{t}&= \phi(-1)K(t) \frac{J(\phi S, \phi K \ol S)}{J(S, K\ol S)}  \pPPq21{\phi K \ol S& S}{&K}t^2 -\phi(1-t) \cdot q \\
		&=  \phi(1-t)  \pPPq21{\phi K \ol S& S}{&K}t  \pPPq21{\phi \ol K S & \ol S}{& \ol K}t -\phi(1-t) \cdot q
		\end{align*}
		
		When $t=1$, according to \cite[Theorem 4.38]{Greene} (or \cite[Equation (7.2)]{Win3X}),  the evaluation is

		\begin{align*}
		\pPPq32{\phi&\eta&\ol \eta}{&K&\ol K}{1}&=  \phi(-1)\pPPq32{\ol\eta K&\eta K&\phi K}{&K^2&K}1\\
		& =\phi(-1)J(\eta K,\ol \eta K)J(\phi,\phi K)\left(\sum_{D=S, \phi S}\frac{J(D,K\ol S^2)J(\phi \ol D,\ol K S^2)}{J(\phi S, \ol K)J(S, \ol K)}\right)
		\end{align*}
		which can be converted to the right hand side of \eqref{eq:3P2(1)} using properties of Jacobi sums. \bk
		The identity (\ref{eq:Clausen-K=1}) can be verified directly, and (\ref{eq:not-square}) follows from Theorem 4.38(ii) of \cite{Greene} by Greene.
	\end{proof}
	
	\begin{lemma}\label{lem:det=qq}For $q$ odd, the product of two summands of \eqref{eq:3P2(1)} equals $q^2$, namely

		\begin{multline*}
		\(\phi\eta(-1)q\frac{J(\eta K,\ol \eta K)J(\phi,\phi K)}{J(\phi S, \ol K)J(S, \ol K)}\right)^2\frac{J(S,\phi \ol S)}{J(S\ol\eta, \phi\ol S\eta)}\frac{J(\phi S,\ol S)}{J(\eta\ol S,\phi S\ol\eta)}=q^2.\\
		\end{multline*}
	\end{lemma}
	\begin{proof}
		It is a straightforward verification. Using the double angle formula 
		$$
		\fg(A)\fg(\phi A)=\fg(A^2)\fg(\phi)\ol A(4),
		$$
		and the reflection formula 
		$$
		\fg(A)\fg(\ol A)=A(-1)q,\quad A\neq \eps
		$$
		on Gauss sums $\fg(A)$,
		we have 
		\begin{multline*}
		\left( \frac{J(\eta K,\ol \eta K)J(\phi,\phi K)}{J(\phi S, \ol K)J(S, \ol K)}\right )^2=\left ( \frac{\fg(\eta K)\fg(\ol \eta K)\fg(\phi S \ol K)\fg(\phi)\fg( \phi K)\fg(S\ol K)}{ \fg(K^2)\fg(\phi S)\fg( \ol K)\fg( K)\fg( S)\fg( \ol K)}\right )^2\\
		=\left ( \frac{\fg(\eta K)\fg(\ol \eta K)\fg(\phi)^2\fg( \phi K)\fg(S^2\ol K^2)}{ \fg(\phi K)\fg( \ol K)^2\fg( K)^2\fg( S^2)}\right )^2
		=\left ( \frac{\fg(\eta K)\fg(\ol \eta K)\fg(\phi)^2\fg( \phi K)\fg(\eta\ol K)}{ \fg(\phi K)\fg( \ol K)^2\fg( K)^2\fg( \eta K)}\right )^2
		=1
		\end{multline*}and
		$$
		\frac{J(S,\phi \ol S)}{J(S\ol\eta, \phi\ol S\eta)}\frac{J(\phi S,\ol S)}{J(\eta\ol S,\phi S\ol\eta)}= \frac{q^2}{ q^2}=1.$$ The claim hence follows.
	\end{proof}
	
	\begin{lemma}\label{lem:6P5(1)}
		For a finite field $\F_q$ of odd characteristic  and  $A$, $B$, $C$, $D$, $E\in \widehat{\F_q^\times}$, we have
		\begin{multline*}
		\pPPq65{A&B &C&A& D& E}{&A\ol D&A \ol E&\eps&A\ol B&A \ol C}{1} =BCDE(-1) \sum_{t\in \F_q} A(t) \pPPq32{A& B&C}{&A\ol D&A \ol E}{t} ^2\\
		= BCDE(-1)\sum_{t \in\F_q} A(t) \pPPq32{A& D&E}{&A\ol B&A \ol C}{t} ^2.
		\end{multline*}
	\end{lemma}
	
	\begin{proof}
		\begin{align*}&BCDE(-1)\sum_{t\in \F_q} A(t)\pPPq32{A&B&C}{&A\ol D&A\ol E}{t}^2\\
		&\overset{Prop. \ref{prop: Kummer}}=A(-1)\sum_{t\in \F_q^\times}\pPPq32{A&B&C}{&A\ol D&A\ol E}{t}\pPPq32{A&D&E}{&A\ol B&A\ol C}{1/t}\\
		&=A(-1)\sum_{t\in \F_q^\times} \frac{-1}{q-1}\sum_{\chi \in \widehat{\F_q^\times}}\binom{A\chi}{\chi}\binom{B\chi}{A\overline D\chi}\binom{C\chi}{A\overline E\chi} \chi(t)\frac{-1}{q-1}\sum_{\psi \in \widehat{\F_q^\times}} \binom{A\psi}{\psi}\binom{D\psi}{A\overline B\psi}\binom{E\psi}{A\overline{C}\psi}\psi(1/t) \\
		& =A(-1)\frac1{(q-1)^2}\sum_{\chi \in \widehat{\F_q^\times}}\binom{A\chi}{\chi}\binom{B\chi}{A\overline D\chi}\binom{C\chi}{A\overline E\chi} \sum_{\psi \in \widehat{\F_q^\times}} \binom{A\psi}{\psi}\binom{D\psi}{A\overline B\psi}\binom{E\psi}{A\overline{C}\psi}\sum_{t\in \F_q^\times}\chi\psi^{-1}(t).
		\end{align*}
		
		By the following orthogonality relation 
		$$
		\sum_{t\in \F_q^\times}\chi\psi^{-1}(t)=
		\begin{cases}
		&0,\quad \mbox{if } \chi\neq \psi,\\
		&q-1,\quad \mbox{if } \chi= \psi,
		\end{cases}
		$$
		we obtain that 
		
		\begin{align*}&BCDE(-1)\sum_{t\in \F_q} A(t)\pPPq32{A&B&C}{&A\ol D&A\ol E}{t}^2\\
		&=A(-1)\frac{1}{q-1}\sum_{\chi  \in \widehat{\F_q^\times}}\binom{A\chi}{\chi}\binom{B\chi}{A\overline D\chi}\binom{C\chi}{A\overline E\chi} \binom{A\chi}{\chi}\binom{D\chi}{A\overline B\chi}\binom{E\chi }{A\overline{C}\chi}\\
		&= \pPPq65{A&B &C&A& D& E}{&A\ol D&A \ol E&\eps&A\ol B&A \ol C}{1}.
		\end{align*}
	\end{proof}
	
	\subsection{Hypergeometric character sums and Galois representations}\label{ss:3.2}
	To a hypergeometric datum $HD=\{\alpha,\beta;\l\}$ consisting of a primitive pair of multi-sets $\alpha=\{a_1,..., a_n\}, \beta=\{1, b_2,..., b_n\}$ with $M = lcd(\alpha \cup \beta)$ and $\l \in \Z[\zeta_M,1/M]\setminus \{0\}$, and a prime $\ell$, we associate the $\ell$-adic representation $\rho_{HD,\ell}$ of $G(M) = G_{\Q(\zeta_M)}$ as in Theorem \ref{thm:Katz}. We are interested in when it can be extended to a representation of $G_K$ for a subfield $K$ of $\Q(\zeta_M)$, in particular $K=\Q$. This happens if and only if $\rho_{HD,\ell}$ is isomorphic to its conjugation by each element in Gal$(\Q(\zeta_M)/K)$. This is a standard result by the Clifford theory \cite{Clifford} and Frobenius reciprocity. For related discussions, see \cite[\S2.6]{LLL} by Li, Liu and Long.  If this happens we say  $\{\alpha,\beta;\l\}$ can be \emph{extended} to $K$.
	
	We now explain the Galois action. An element $\tau \in$ Gal$(\Q(\zeta_M)/\Q)$ sends $\zeta_M$ to $\zeta_M^c$ for some $c = c(\tau) \in (\Z/M\Z)^\times$. The  conjugation of $\rho_{HD,\ell}$ by $\tau$, denoted by  $\rho_{HD,\ell}^\tau$, satisfies
	$$ \Tr \rho_{HD,\ell}^{\tau}(\Frob_\wp) = \Tr \rho_{HD, \ell}(\Frob_{\tau(\wp)}) = \Tr \rho_{c\cdot HD, \ell}(\Frob_\wp)$$
	for almost all prime ideals $\wp$ of $\Z[\zeta_M]$. Regard the pair $\alpha, \beta$ in HD as a multi-set of column vectors $\left\{\begin{pmatrix}a_1\\1\end{pmatrix},\cdots, \begin{pmatrix}a_n\\b_n\end{pmatrix} \right\}$. Here $c\cdot HD$ denotes the datum 
	$$\{c\cdot \alpha, c\cdot \beta; \l\}= \left\{\begin{pmatrix}ca_1\\c \end{pmatrix},..., \begin{pmatrix}ca_n\\cb_n \end{pmatrix}; \l\right\}.$$ 
	This motivates the following definitions.
	
	\begin{definition}\label{def:1}
		For a primitive pair $\alpha, \beta$ with $M=lcd(\alpha \cup \beta)$ as above and $c \in (\Z/M\Z)^\times$, the pairs $\alpha, \beta$ and $c\cdot \alpha, c\cdot \beta$ are \emph{congruent mod $\Z$} if 
		the two multi-sets 
		$\left\{\begin{pmatrix}ca_1\\c \end{pmatrix},\begin{pmatrix}ca_2\\cb_2\end{pmatrix},...,   \begin{pmatrix}ca_n\\cb_n \end{pmatrix}\right\}$ and $\left\{\begin{pmatrix}a_1\\1 \end{pmatrix}, \begin{pmatrix}a_2\\b_2 \end{pmatrix},...,   \begin{pmatrix}a_n\\b_n \end{pmatrix}\right\}$ are congruent mod $\Z$. The pair $\alpha, \beta$ is said to be \emph{self-dual} if it is congruent to the pair $-\alpha, -\beta$ mod $\Z$, and it is \emph{defined over a subfield $K$ of $\Q(\zeta_M)$} if it is congruent to $c\cdot \alpha, c\cdot \beta$ mod $\Z$ for all $c=c(\tau)$, where $\tau \in$ Gal$(\Q(\zeta_M)/K)$. A hypergeometric datum $HD=\{\alpha, \beta; \l\}$ is said to be \emph{self-dual} if the pair $\alpha, \beta$ is self-dual and $\l$ is totally real; it is \emph{defined over $\Q$} if the pair $\alpha, \beta$ is defined over $\Q$ and $\l \in \Q^\times$.
	\end{definition}

	Clearly, if the pair $\alpha, \beta$ in $HD$ above is defined over $K$  and $\l\in K^\times$,  then the representation $\rho_{HD,\ell}$ can be extended to $G_K$.

	\begin{remark}\label{rk:density} Let $K$ be a number field  with the ring of integers $\Z_K$. Since the prime ideals of $\Z_K$ of degree  1  have density one  among all prime ideals of $\Z_K$, by the Chebotarev density theorem, if two  $\ell$-adic representations of $G_K$ ramified only at finitely many prime ideals have the same traces at $\Frob_\wp$ for almost all prime ideals  $\wp$ of $\Z_K$ of degree 1, then they are isomorphic as $G_K$ representations provided that they are semi-simple. In particular, if $K = \Q(\zeta_M)$ is a cyclotomic field, the semi-simplification of an $\ell$-adic representation of $G_K$ ramified only at finitely many prime ideals, such as those in Theorems \ref{thm:Katz} and \ref{thm:KBCM}, is determined by its  traces at $\Frob_\wp$ for almost all prime ideals $\wp$ of $\Z_K$ above primes $p \equiv 1 \mod M$.
	\end{remark}

	\medskip

	\subsection{BCM representations}\label{sec:BCM} 
	When $HD=\{\alpha, \beta; \l\}$ with $M=lcd(\alpha \cup \beta)$ is defined over $\Q$, 
	the representation  $\rho_{HD,\ell}$ of $G(M)$ can be extended to $G_\Q$. The extensions differ by  twists by characters of $G_\Q$ trivial on $G(M)$. The next result provides such an extension expressed in terms of the hypergeometric sums  $H_q(HD)$  defined in \cite[Definition 1.1]{BCM} by Beukers, Cohen and Mellit using Gauss sums  (see \eqref{eq:hq}  for the case when $HD$ is defined over $\Q$).  They depend only on the elements in $\alpha$ and $\beta$ mod $\Z$ and are independent of their order. For $\alpha =\{a_1,..., a_n\}$, $\beta = \{1, b_2,..., b_n\}$ forming a primitive pair  and $\l$ a nonzero element in a finite field $\F_q$ containing $M$th roots of $1$, $H_q$ and $\P$ are related by  
	\begin{equation}\label{eq:P-H}
	\P(\alpha,\beta;\l;\F_q;\omega) =H_q(\alpha,\beta;\l;\omega) \cdot \prod_{i=2}^n\omega^{(q-1)a_i}(-1)J(\omega^{(q-1)a_i},\omega^{(q-1)(-b_i)}),
	\end{equation}
	where
	$J(A,B)=\sum_{x\in \F_q}A(x)B(1-x)$ 
	is the Jacobi sum of characters $A$ and $B$.
	
	Given a hypergeometric datum $HD'=\{\alpha', \beta'; 1\}$, where $\alpha'=\{c, a\}, \beta'=\{1, b\}$ is a primitive pair with $lcd(\alpha'\cup \beta') = M'$, for a finite field $\F_q$ containing $M'$th roots of $1$,  we have 
	\begin{equation}\label{eq:Gauss}\P(HD'; \F_q; \omega) = J(\omega^{(q-1)a},\omega^{(q-1)(b-c-a)})
	\end{equation}(from either the definition of the $\P$-function or (6.11) of \cite{Win3X}).  By Theorem \ref{thm:Katz}, the representation $\rho_{HD', \ell}$ is a character of $G(M')$. 
	The relation (\ref{eq:P-H}) shows that when the elements in the multi-sets $\alpha$ and $\beta$ are permuted, the associated representations in Theorem \ref{thm:Katz} differ by a character of $G(M)$. On the other hand, the representation remains the same if the elements in $\alpha$ and $\beta$ are shuffled by the same permutation on $\{2,..., n\}$. For this reason, we shall regard the pair $\alpha, \beta$ in HD as a multi-set of column vectors $\left\{\begin{pmatrix}a_1\\1\end{pmatrix},\cdots, \begin{pmatrix}a_n\\b_n\end{pmatrix} \right\}$ from now on.
	
	Since the Jacobi sum $J(A, B)$ for any $A, B \in \widehat{\F_q^\times}$ has absolute value $q^{1/2}$ as long as $A$, $B$ and $AB$ are nontrivial, and is equal to $-1$ if one of $A, B$ is trivial, for the primitive pair $\alpha, \beta$ above, we have
	\begin{equation}\label{eq:J}
	\mathcal{J}(\alpha, \beta;\F_q; \omega) := \prod_{i=2}^n {\omega^{(q-1)a_i}(-1)}J(\omega^{(q-1)a_i},\omega^{(q-1)(-b_i)})= (-1)^{m-1}\chi(\alpha, \beta; \F_q) q^{(n-m)/2},
	\end{equation}
	where $m$ is the number of elements in $\beta$ which are in $\Z$ and $\chi(\alpha, \beta; \F_q)$ is a root of unity. If, moreover, the pair $\alpha, \beta$ is self-dual (see \S\ref{ss:3.2} for definition), then {$\chi(\alpha, \beta; \F_q)=\prod_{i=2}^n\omega^{(q-1)a_i}(-1) = \pm 1$ and $n-m$ is even.}
	\medskip

	Moreover, if $HD = \{\alpha, \beta; \l\}$ is defined over $\Q$, writing $\displaystyle \prod_{j=1}^n \frac{X-e^{2\pi i a_j}}{X-e^{2\pi i b_j}}=\frac{\prod_{j=1}^r(X^{p_j}-1)}{\prod_{k=1}^s(X^{q_k}-1)}$ where $p_j,q_k\in \Z_{>0}$ and $p_j \ne q_k$ for all $j, k$,  by Theorem 1.3 of \cite{BCM}, the character sum $H_q(\alpha,\beta;\l)$ 
	can be defined in the following way for any finite field $\F_q$ as long as $q$ is coprime to $lcd(\{1/M,\l\})$ 
	\begin{equation}\label{eq:hq}
	H_q(\alpha,\beta;\l):=\frac{(-1)^{r+s}}{1-q}\sum_{m=0}^{q-2}q^{-s(0)+s(m)}\prod_{j=1}^r \mathfrak{g}(\omega^{mp_j}) \prod_{k=1}^s \mathfrak{g}(\omega^{-mq_k})\omega(N^{-1} \l),
	\end{equation}where $\mathfrak{g}(\chi)$ denotes the Gauss sum of the character $\chi$, $N=(-1)^{ q_1+\cdots +q_s} \frac{p_1^{p_1}\cdots p_r^{p_r}}{q_1^{q_1}\cdots q_s^{q_s}}$ and $s(m)$ is the multiplicity of $X-e^{2\pi i m/(q-1)}$ in the factorization of $\displaystyle gcd(\prod_{j=1}^r(X^{p_j}-1),\prod_{k=1}^s(X^{q_k}-1)).$ 
	Further, since $H_q$ is $\Q$-valued, it is independent of the choice of $\omega$.
	\medskip

	\begin{theorem}[Katz, Beukers-Cohen-Mellit]\label{thm:KBCM} Let $\alpha=\{a_1,\cdots, a_n\},\beta=\{1,b_2,\cdots, b_n\}$ be a primitive pair with $M=lcd(\alpha\cup \beta)$  and let $ \l\in \Z[1/M]\smallsetminus \{0\}$. Suppose the hypergeometric datum $HD =\{\alpha, \beta; \l\}$ is defined over $\Q$  (see Definition \ref{def:1}).  Assume that exactly $m$ elements in $\beta$ are in $\Z$. 
		Then, for each prime $\ell$, there exists an $\ell$-adic  representation $\rho_{HD,\ell}^{BCM}$ of $G_\Q$ with the following properties:
		\begin{itemize}
			\item[i).] $\rho_{HD,\ell}^{BCM}|_{G(M)}\cong \rho_{HD,\ell}$.
			\item[ii).] For any prime  $p\nmid \ell \cdot M$ such that  $\ord_p \l = 0$, 
			\begin{equation}
			\Tr\, \rho_{HD,\ell}^{BCM}(\Frob_p)= \phi(M,a_1)(\Frob_p) {\chi(\alpha, \beta; \F_p)}  
			H_p(\alpha,\beta;1/\l)\cdot p^{(n-m)/2}\in \Z.
			\end{equation} 
			\item[iii).]  When $\l=1$,   $\rho_{HD,\ell}^{BCM}$ is $(n-1)$-dimensional and it has a subrepresentation, denoted by $\rho_{HD,\ell}^{BCM,prim}$, of dimension $2\lfloor \frac {n-1}2 \rfloor$ whose representation space admits a symmetric (resp. alternating) bilinear pairing if $n$ is odd (resp. even).  All roots of the characteristic polynomial of $\rho_{HD,\ell}^{BCM,prim}(\Frob_p)$ have absolute value $p^{(n-1)/2}$. 
		\end{itemize}
	\end{theorem}
	
	Note that $n-m$ is even, resulting from the primitive and self-dual assumptions on the pair $\alpha, \beta$. 
	
	The character $\phi(M,a_1)$ in Remark \ref{phi(M,a_1)}  
	extends to a character of $G_\Q$. The character $\phi(M,a_1)$ in Theorem \ref{thm:KBCM} refers to the extension with minimal conductor. When it is nontrivial, that is, when $\ord_2~M = -\ord_2 ~a_1 = r\ge 1$, it has conductor $2^{r+1}$. 
	In particular, when $\ord_2~M = -\ord_2~a_1 = 1$, $\phi(M,a_1)(\Frob_p)=\(\frac{-1}{p}\)$ is given by the Legendre symbol at odd primes $p$.

	For the pairs $(c,f)$ in (\ref{eq:7pair}), the characters $\chi$ and $\phi$ in Theorem \ref{thm:KBCM} at unramified odd primes $p$ are given by $\phi(M(c,f), a_1)(\Frob_p) = \(\frac{-1}{p}\) = \chi(\alpha_6(c,f), \beta_6(c,f); \F_p) = \chi(\alpha_4(f), \beta_4(c); \F_p)$ so that 
	\begin{eqnarray}\label{eq:BCMtrace}
	\Tr\, \rho_{HD_i(c,f),\ell}^{BCM}(\Frob_p)= 
	H_p(HD_i(c,f))\cdot p^{(n-m)/2} \quad {\rm for}~~i=1,2.
	\end{eqnarray}

	\medskip
	\subsection{Modularity}
	To obtain the modularity result for degree-2 Galois representations of $G_\Q$, we shall use the following version of the modularity theorem; see Theorem 2.1.4 of \cite{ALLL}.
	
	\begin{theorem}\label{thm:modulairty}
		Given a prime $\ell$ and a 2-dimensional absolutely irreducible representation $\rho$ of $G_\Q$ over $\overline \Q_\ell$ that is odd, unramified at almost all primes, and its restriction to a decomposition subgroup $D_\ell$ at $\ell$ is crystalline with Hodge-Tate weight $\{0,r\}$ where $1\le r\le \ell-2$ and $\ell+1\nmid 2r$, then $\rho$ is modular and corresponds to a weight $r+1$ holomorphic Hecke eigenform. \end{theorem}
	
	The actual identification of the target modular form is carried out by computing the trace of the representation at various Frobenius conjugacy classes $\Frob_p$. In the next subsection we discuss  key issues involved in computation. Since representations from Theorem \ref{thm:KBCM} have integral traces, we use  
	the following theorem of Serre (cf. \cite[\S4.8]{Serre87} by Serre or \cite[Theorem 2.2]{Dieulefait04} by  Dieulefait) to narrow the search for the corresponding modular forms.
	\begin{theorem}[Serre]\label{thm:Serre} Suppose the trace of the representation $\rho$ of $G_\Q$ in Theorem \ref{thm:modulairty} is $\Z$-valued.
		Then  the $p$-exponents of the  conductor  of $\rho$
		are bounded by $8$ for $p=2$, by $5$ for $p=3$, and by $2$ for all other bad primes.
	\end{theorem}
	
	\subsection{Some computational issues}\label{ss:3.3}
	Given $\alpha =\{a_1,..., a_n\}$ and $\beta = \{b_1,..., b_n\}$ with $a_i, b_j \in \Q\cap [0, 1)$ and $M = lcd(\alpha \cup \beta)$, the following step function on the interval $[0, 1)$ is introduced in \cite{Long18} by the second author:  
	$$e_{\alpha,\beta}(x):=\sum_{i=1}^n -\left \lfloor a_i -x\right \rfloor-\left \lfloor x+ b_i  \right \rfloor.
	$$
	The  value of $e_{\alpha,\beta}(x)$ jumps up (resp. down)  only  at  $a_i$ (resp. $1-b_j$).
	When $p$ is an odd prime not dividing $M$, $0 \le k <p-1$ is an integer and the pair $\alpha, \beta$ is defined over $\Q$, $e_{\alpha,\beta}(\frac{k}{p-1})$ gives the collective exponent of $p$ in the $k$th summand of $H_p(\alpha,\beta;\l)$ recalled in \eqref{eq:hq}. 
	Moreover when the pair $\alpha, \beta$ is defined over $\Q$,  the graph of  $e_{c\alpha,c\beta}(x)$ is independent of any integer $c$ prime to $M$.

	The weight $w(HD)$ of a datum $HD=\{\alpha, \beta;\l\}$ is defined as 
	\begin{equation}\label{eq:w}
	w(HD)=w(\alpha,\beta):=\max e_{\alpha,\beta}(x)-\min e_{\alpha,\beta}(x).\end{equation}  When $\rho_{HD,\ell}$ can be extended to $G_\Q$,  $w(c\alpha,c\beta)$ is independent of $c$, although the graph of  $e_{c\alpha,c\beta}(x)$ may {depend on $c$}.

	From the definition, one can implement $H_p(\alpha,\beta;\l)$ in any computational package of choice.  In some of our computation and posted data, we use \texttt{Sagemath}. 
	For $HD$ defined over $\Q$, there is an efficient \texttt{Magma} program called ``Hypergeometric Motives over $\Q$" implemented by Watkins which computes the characteristic polynomial of $\rho_{\{\alpha,\beta;\l\},\ell}^{BCM}[t]$ (resp. $\rho_{\{\alpha,\beta;\l\},\ell}^{BCM,prim}[t]$) at $\text{Fr}_p$, the inverse of $\Frob_p$, for $p \nmid M\ell$ efficiently when $\l \neq 0,1$ (resp. $\l=1$), where 
	\begin{equation}\label{eq:t}
	t=-\min \{e_{\alpha,\beta}(x) \mid 0\le x< 1\}- \frac{n-m}2,
	\end{equation} which is an integer. We let
	$\rho[t]$ denote the weight-$t$ Tate twist of a representation $\rho$ of $G_\Q$.
	\medskip
	
	\begin{example}

		For $\alpha=\{\frac12,\frac15,\frac25,\frac35,\frac45\}$, $\beta=\{1,\frac1{10},\frac3{10},\frac7{10},\frac9{10}\}$, the graph of the function is as follows.
		
		\begin{center}
			\includegraphics[scale=0.3,origin=c]{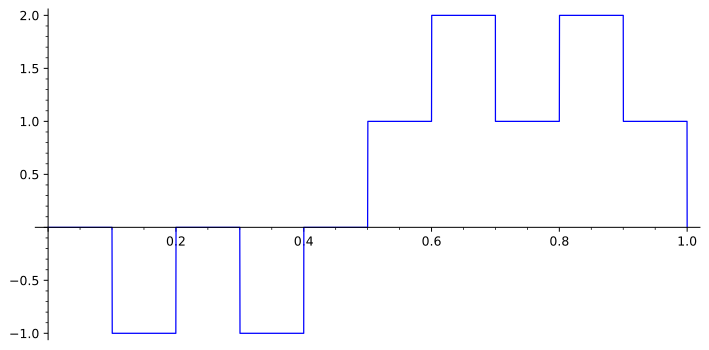}
			
			The graph of $e_{\alpha,\beta}(x)$ with $\alpha=\{\frac12,\frac15,\frac25,\frac35,\frac45\}$, $\beta=\{1,\frac1{10},\frac3{10},\frac7{10},\frac9{10}\}$ 
		\end{center} In this case, $\min e_{\alpha,\beta}=-1,$ $\max e_{\alpha,\beta}=2$, $w(HD)=2-(-1)=3$, $n=5,m=1$, so $t=-\min e_{\alpha,\beta}- \frac{n-m}2=1-\frac{5-1}2=-1$. When $\l=1$, $\dim \rho_{HD,\ell}^{BCM,prim}=2\lfloor \frac{5-1}2\rfloor=4$. 
		We exhibit below the commands used in the \texttt{Magma} program to compute the characteristic polynomial of  $\rho_{HD,\ell}^{BCM,prim}[-1]$ at $\text{Fr}_7$:
		\\
		
		\texttt{H:=HypergeometricData([1/2,1/5,2/5,3/5,4/5],[1,1/10,3/10,7/10,9/10]);}
		
		\texttt{Factorization(EulerFactor(H,1,7));}
		
		\medskip
		
		The output is
		
		\texttt{$<2401*\$.1^4 - 88*\$.1^2 + 1, 1>$}
		
		\medskip
		
		It confirms that that $\dim \rho_{HD,\ell}^{BCM,prim}=4$ and the roots of the characteristic polynomial of $\rho_{HD,\ell}^{BCM,prim}[-1]$ at $\text{Fr}_7$ all have absolute value $1/7$ under any embedding, which is $1/p^{(n-1)/2+t}$ at $p=7$. 
	\end{example}
	
	\begin{example}
		We now explain how to identify the weight-4 modular form corresponding to $HD_2(\f12,\f12).$
		Using \texttt{Magma}, we obtain the following output of the first few characteristic polynomials of $\text{Fr}_p$.
		
		$
		[ 3,\, 27*\$.1^2 + 4*\$.1 + 1],$
		
		$    [
		5,\,
		125*\$.1^2 + 2*\$.1 + 1
		],$
		
		$   [
		7,\,
		343*\$.1^2 - 24*\$.1 + 1
		],$
		
		$    [
		11,\,
		1331*\$.1^2 + 44*\$.1 + 1
		],$
		
		From the list we know the corresponding modular form has weight 4. As $lcd(HD_2(\f12,\f12))=2$, the level of the modular form is a power of 2 from which we identify the target modular form  $$f_{8.4.a.a}(\tau)=q-4q^3-2q^5+24q^7-11q^9-44q^{11}+\cdots.$$   
	\end{example}
	\begin{example}

		Here we consider another example  $HD=\{\{\frac13,\frac23\},\{1,\frac12\};-1\}$, in which each of $\{\frac13,\frac23\}$ and $\{1,\frac12\}$ is  defined over $\Q$ as a multi-set, but it is not so as a pair (see Definition \ref{def:1}).   
		
		\smallskip
		A partial  output is

		$    [        7,   \,     \$.1^2 + \$.1 + 1],$
		
		$    [        11,   \,     \$.1^2 + \$.1 + 1],$
		
		$    [        13,  \,      -\$.1^2 + 1],$
		
		$    [        17,  \,      -\$.1^2 + 1],$
		
		$    [        19,  \,      -\$.1^2 + 1],$
		
		$    [        23,   \,     -\$.1^2 + 1],$
		
		$    [        29,  \,      \$.1^2 - 2*\$.1 + 1]$
		
		From the above data we first read off the determinant being  $\chi_{-6}$. Looking more closely we see that it corresponds to the weight-1 Hecke eigenform $$f_{216.1.h.b}(\tau)= q + q^{2} + q^{4} - q^{5} - q^{7} + q^{8} - q^{10} - q^{11} - q^{14} + q^{16} - q^{20} - q^{22} - q^{28} + 2q^{29}+\cdots. $$ More explicitly, by Theorems 8.4 and 8.11 of \cite{Win3X} when $p\equiv 1\mod 3$ 
		\begin{align*}
		H_p\(\{\frac13,\frac23\},\{1,\frac12\};-1; \omega \)=& \ol\omega^{\frac{p-1}3}(2)H_p\(\{\frac13,\frac56\},\{1,\frac12\};1/2; \omega \)\\
		=&
		\begin{cases}
		0, &\, \mbox{if } p\equiv 13,19\mod 24,\\
		\ol\omega^{\frac{p-1}3}(2+a)+\ol\omega^{\frac{p-1}3}(2-a),& \, \mbox{if } p\equiv 1, 7 \mod 24,
		\end{cases}
		\end{align*}where $a$ is a fixed square root of $2$ in $\F_p$;
		and  from the conversion between $\P$ and $H_p$ we get: for $p\equiv 1\mod 3$ 
		$$
		\P\(\{\frac13,\frac23\},\{1,\frac12\};-1;p;  \omega \)=J(\ol\omega^{\frac{p-1}3},\omega^{\frac{p-1}2}) H_p\(\{\frac13,\frac23\},\{1,\frac12\};-1;\omega\).
		$$It indicates that $\rho_{HD,\ell}$, which is a representation of $G(3)$, can not be directly extended to $G_\Q$; however, after twisting by a suitable character of $G(3)$, an extension exists. \bk
	\end{example}

	\medskip
	
	\section{{Galois representations associated to hypergeometric data}}\label{ss:Galois}
	
	\subsection{The Whipple $_7F_6(1)$ formula}\label{sec:Whipple} We  explain how the special form of the Whipple formula in \S \ref{sec:purpose} is derived. The well-known 
	Whipple $_7F_6(1)$ formula 
	asserts that 
	{\tiny\begin{multline}\label{eq:Whipple}
		\pFq{7}{6}{a&1+\frac a2&c&d&e&f&g}{&\frac a2& 1+a-c&1+a-d&1+a-e&1+a-f&1+a-g}{1}\\=
		\frac{\G(1+a-e)\G(1+a-f)\G(1+a-g)\G(1+a-e-f-g)}{\G(1+a)\G(1+a-f-g)\G(1+a-e-f)\G(1+a-e-g)} \cdot\pFq{4}{3}{a&e&f&g}{&e+f+g-a&1+a-c&1+a-d}{1},
		\end{multline}}
	
	\noindent when both sides  terminate;  see Theorem 3.4.4 of the book \cite{AAR} by Andrews, Askey and Roy for backgrounds and details.   The  $_7F_6(1)$ on the left hand side is  well-posed, meaning the upper and lower parameters sum to the same value (equal to $1+a$ in our case) in each column.  By Bailey \cite{Bailey}, the left hand side can be also computed by a contour integral of Barnes type, which is related to the irrationality of $\zeta(3)$, see \cite{Zudilin-ICCM} by Zudilin for the related discussions. To the corresponding hypergeometric systems, Katz \cite{Katz} attached  $\ell$-adic sheaves on the projective line with three points $0, 1, \infty$ removed. 
	For the hypergeometric sheaves arising from the data in the Whipple formula to be self-dual, we impose 
	additional assumptions: $$a=\frac12,\quad   c+d=1, \quad  f+g=1, \quad  \text{and } e=\frac{1-p}2$$ for an odd natural number $p$ so that both sides converge.\footnote{ Here $e$ is set to be $\frac{1-p}2$ instead of the target $\frac12$ so that the Gamma quotient on the right hand side of \eqref{eq:Whipple} is a finite value.}   Whipple's formula now becomes \eqref{eq:Whipple2} in the Introduction.

	Recall that the hypergeometric data $HD_1$ and $HD_2$ are given by 
	\begin{equation*}
	HD_1(c,f)=\left \{\alpha_6(c,f):=\left\{\f12,c,1-c,\f12,f,1-f\right \},\beta_6(c,f):=\left \{ 1,\f32-c,\f12+c,1,\f32-f,\f12+f\right \};1 \right\}
	\end{equation*}
	\begin{equation*}
	HD_2(c,f)=\left\{\alpha_4(f):=\lc\f12,\f12,f,1-f\rc,\beta_4(c):=\lc 1,1,\f32-c,\f12+c\rc;1 \right\}\end{equation*}
	with $c,f\in \Q\cap (0,1)$ as in \S1.2. They are both self-dual from the construction. Let $$M = M(c,f) = lcd(\alpha_6(c,f) \cup \beta_6(c,f)) = lcd(\alpha_4(c,f) \cup \beta_4(c,f)).$$ 
	Then both $\alpha_6(c,f)$
	and $\beta_6(c,f)=1+\f12 -\alpha_6(c,f)$
	in $HD_1$ are defined over $\Q(\zeta_M+\zeta_M^{-1})$. In particular, only 14  un-ordered couples $(c,f)$ give rise to $HD_1$ defined over $\Q$: $$\text{either } c,f\in \left \{\f12,\frac13,\frac14,\frac16\right \}, \quad \text{ or} \quad (c,f)=\left(\frac 15,\frac25\right),\left(\frac 18,\frac38\right), \left(\frac 1{10},\frac3{10}\right)$$ as enumerated in \cite{LTYZ}. Among them, only  seven pairs of $(c,f)$ give rise to primitive $HD_1$,  as listed in \eqref{eq:7pair}.

	\subsection{Proof of  Theorems \ref{thm:LLT2main} and \ref{thm:main}} \label{3.2}
	Let 
	\begin{equation}\label{eq:HD3}
	\alpha_3(f)= \left\{\f12,f,1-f\right\}, \quad \beta_3(c)=\left\{1,\frac32-c,\f12+c\right\}
	\end{equation} which form a primitive pair with $M(c,f)=lcd(\alpha_3(f) \cup \beta_3(c))$.
	Motivated by the finite field Clausen formula (Theorem \ref{thm:E-G}), we define the multi-sets 
	\begin{equation*}
	\alpha_2(c,f):=\left\{ \frac{1+2f-2c}{4},\frac{3-2f-2c}4,\right\}, \quad \beta_2(c):=\left\{1,\frac32-c\right\}
	\end{equation*} with $N(c,f)=lcd(\alpha_2(c,f)\cup \beta_2(c))$.  Then either $N(c,f)=2M(c,f)$ or $N(c,f)=M(c,f)$.

	Fix a prime $\ell$. 
	By Theorem \ref{thm:Katz}, for any nonzero $\l \in  \Z[\zeta_{M(c,f)},1/M(c,f)]$  there exists an $\ell$-adic representation
	\begin{equation}\label{eq:rho}
	\rho_{\{\alpha_3(f),\beta_3(c);\l\},\ell}: G(M(c,f))\rightarrow GL_{\overline {\Q}_\ell}(W_\l(c,f))
	\end{equation}such that at any prime ideal $\wp$  of $\Z[ \zeta_{M(c,f)},1/(N(c,f)\ell)]$ with $\ord_\wp \l = 0$ 
	\begin{equation}\label{eq:trace-cond}
	\text{Tr}  \rho_{\{\alpha_3(f),\beta_3(c),\l\};\ell}(\Frob_\wp) = \phi_{\wp}(-1) \ \P\left (\alpha_3(f),\beta_3(c);{1/\l};\kappa_\wp\right ).
	\end{equation}
	Moreover, $\dim W_1(c,f)=2$ and $\dim W_\l(c,f)=3$  when $\l\neq 0,1$.

	We first deal with the $\l=1$ case. {Let $HD_3(c,f)=\{\alpha_3(f),\beta_3(c);1\}$, and write 
		$\rho_{HD_3(c,f),\ell}$ for the $2$-dimensional representation  $\rho_{\{\alpha_3(f),\beta_3(c),1\};\ell}$.

		Let $\wp$ be a prime ideal  of $\Z[ \zeta_{N(c,f)},1/(N(c,f)\ell)]$. A similar formula for $\text{Tr}  \rho_{HD_3(c,f);\ell}(\Frob_\wp)$ holds. The character $\omega_{\wp}^{(N(\wp) -1)(f+\frac32-c)}$ is a square in $\widehat{\kappa_\wp^\times}$ by the definition of $N(c,f)$. Hence we can apply Theorem \ref{thm:E-G} to get more information on $\P\left (\alpha_3(f),\beta_3(c);1;\kappa_\wp\right )$.} 
	It follows from (\ref{eq:3P2(1)}) and (\ref{eq:Gauss}) that    $\rho_{HD_3(c,f),\ell}|_{G(N(c,f))}$ is a direct sum of two explicit characters $\xi_1$ and $\xi_2$ of $G(N(c,f))$ associated to Jacobi sums, see \cite{Weil52}. Further, Lemma \ref{lem:det=qq} implies $\det \rho_{HD_3(c,f),\ell}|_{G(N(c,f))} = \epsilon_\ell^2$. Since $G(N(c,f))$ has index at most $2$ in $G(M(c,f))$, there is a character $\psi_{c,f}$ of $G(M(c,f))$ of order dividing $2$ such that $\det \rho_{HD_3(c,f),\ell} = \psi_{c,f}\epsilon_\ell^2$. We record this in
	
	\begin{prop}\label{prop:detrho3} There is a character $\psi_{c,f}$ of $G(M(c,f))$ of order dividing $2$ such that $$\det \rho_{HD_3(c,f),\ell} = \psi_{c,f}\epsilon_\ell^2.$$ Moreover, $\det \rho_{HD_3(c,f),\ell}|_{G(N(c,f))} = \epsilon_\ell^2$. 
	\end{prop}
	
	For the seven pairs $(c,f)$ listed in \eqref{eq:7pair}, we have more detailed information about the representation $\rho_{HD_3(c,f),\ell}$, as shown below. Among them, $N(c,f)=M(c,f)$ only happens when $(c,f)=(\f12,\frac13).$

	\begin{prop}\label{cor:3P2lift}Let $\ell$ be a prime. For each pair $(c,f)$ in \eqref{eq:7pair}, 
		$\rho_{HD_3(c,f),\ell}$ is invariant under the twist by the quadratic character $\chi_{-1}|_{G(M(c,f))}$  when $(c,f) \ne (\frac12, \frac13)$. 
		Further, each $\rho_{HD_3(c,f),\ell}$ can be extended to a representation of $G_\Q$, which is modular and invariant under a quadratic twist.
		Each extension has determinant $\psi_{c,f}\epsilon_\ell^2$, where $\psi_{c,f}$ is a quadratic character of $G_\Q$ given in the table below, extending the character $\psi_{c,f}$ of $G(M(c,f))$ in Prop. \ref{prop:detrho3}. 
		Listed below are the modular forms corresponding to $\rho_{HD_3(c,f),\ell}^{BCM}$ when $HD_3(c,f)$ is defined over $\Q$ and one of the extensions for the last two cases. 
		$$
		\begin{array}{|c|c|c|c|c|}
		\hline
		(c,f) & w(HD_3)&\Tr \rho_{\{\alpha_3(f),\beta_3(c);1\},\ell}(\Frob_p)&  \det \rho_{\{\alpha_3(f),\beta_3(c);1\},\ell}=\psi_{c,f}\epsilon_\ell^2 & CM \\ \hline
		(\f12,\frac12) &3 &a_p(\eta(4\tau)^6)&\chi_{-1}\epsilon_\ell^2& \chi_{-1}\\ \hline
		\(\frac 12,\frac13\)&3&
		{ a_p(\eta(2\tau)^3\eta(6\tau)^3)}&\chi_{-3}\epsilon_\ell^2& \chi_{-3}\\
		\hline
		(\frac13,\frac13)&3 &a_p(f_{36.3.d.b})&\chi_{-1}\epsilon_\ell^2&\chi_{-1} \\ \hline
		(\f12,\frac16)&3&\( \frac{-3}p\) a_p(\eta(4\tau)^6)&\chi_{-1}\epsilon_\ell^2&\chi_{-1} \\
		\hline
		(\frac16,\frac16) &1& p\cdot a_p(\eta(12\tau)^2)&\chi_{-1}\epsilon_\ell^2&\chi_{-1},   \chi_{-3}\\
		\hline
		(\frac15,\frac25)    &3 & a_p(f_{20.3.d.a})&\chi_{-5}\epsilon_\ell^2&\chi_{-5}\\
		\hline
		(\frac1{10},\frac3{10})&1& p\cdot a_p(\eta(4\tau)\eta(20\tau))&\chi_{-5}\epsilon_\ell^2& \chi_{-1}, \chi_{-5}\\ \hline
		\end{array}$$
		
	\end{prop} 
	
	\begin{proof} Write $N$ for $N(c,f)$ and $M$ for $M(c,f)$ when there is no ambiguity. As noted above, 
		$\rho_{HD_3,\ell}|_{G(N)}$ is a direct sum of two  
		characters $\xi_1$ and $\xi_2$ of $G(N)$.  When $(c,f)\neq (\frac 16,\frac16)$ or $(\frac 1{10},\frac3{10})$, $\xi_1$ is not equal to $\xi_2$. \bk Furthermore, by \eqref{eq:not-square}, when $N = 2M$, that is, $(c,f) \ne (\f12,\frac13)$, $\P(HD_3,q)=0$ when $q\equiv 3 \mod 4$.  
		This implies that $\rho_{HD_3,\ell}$ is invariant under twist by $\chi_{-1}|_{G(M)}$.

		For the first five pairs, $HD_3$ is defined over $\Q$. The determinant of $\rho_{HD_3,\ell}^{BCM}$, by Theorem \ref{thm:KBCM}, can be determined explicitly as it is $\epsilon_\ell^2$ times a character of order at most 2 and conductor dividing $24$. They turn out to be either $\chi_{-1}\epsilon_\ell^2$ or $\chi_{-3}\epsilon_\ell^2$, as listed in the table.  For example, when $(c,f)=(\f12,\frac13)$, the determinant at $\Frob_p$ equals $-25,49,-121,169$ for $p=5,7,11,13$, respectively. One sees that the determinant is $\chi_{-3}\epsilon_\ell^2$. 
		This shows that  $\rho_{HD_3,\ell}^{BCM}$ is odd.  When $(c,f)\neq (\frac16,\frac16)$, from $\xi_1\neq \xi_2$ we know $\rho_{HD_3,\ell}^{BCM}$ is irreducible. 
		When $(c,f) = (\frac 16,\frac16)$,  by \eqref{eq:hq} and the  reflection formula for Gauss sum, we have for any prime $p$ not dividing $6$ 
		$$
		H_p(HD_3) =  \frac{1}{p-1}\sum_{\chi\in \widehat{\F_q^\times}}\frac{\mathfrak{g}(\chi^6)\mathfrak{g}(\ol{\chi}^3)}{\mathfrak{g}(\chi^3)}\chi(-2^{-6}).
		$$
		When $p\equiv 2 \mod 3$, the map $\chi \mapsto \chi^3$ gives an automorphism of $\widehat{\F_p^\times}$. It follows that 
		\begin{align*}
		H_p(HD_3)  =&\frac{1}{p-1}\sum_{\chi}\frac{\fg(\chi^2)\fg(\ol{\chi})}{\fg(\chi)}\chi(-2^{-2})
		= \frac{1}{p-1}\sum_{\chi}\frac{\fg(\chi)\fg(\phi_p\chi)\fg(\ol{\chi})}{\fg(\phi_p)\fg(\chi)}\chi(-1)\\
		=&\frac{1}{p-1}\sum_{\chi} J(\phi_p\chi,\ol{\chi})\chi(-1)=\frac{1}{p-1}\sum_x \phi_p(x)\sum_{\chi}\chi(-x)\ol{\chi}(1-x)=0.
		\end{align*} To obtain the second  equality above we use the duplication formula for Gauss sums. 
		If $\rho_{HD_3,\ell}^{BCM}$ is reducible, then it is isomorphic to $\xi\oplus \xi\chi_{-3}$ for some 1-dimensional  representation $\xi$ of $G_\Q$. Meanwhile, as  $\rho_{HD_3,\ell}$ is invariant when twisted by $\chi_{-1}$, we have 
		$(\xi\oplus \xi\chi_{-3})|_{G(6)}\cong (\xi\oplus \xi)|_{G(6)}\cong(\xi\chi_{-1}\oplus \xi\chi_{-1})|_{G(6)}$, which is impossible.

		For the remaining two pairs, we have  
		$\rho_{HD_3,\ell}\cong (\chi_{-1})|_{G(M)}\otimes \rho_{HD_3,\ell}$ as noted before. In these two cases  $M = 10$ and the field $\Q(\zeta_{M})$ has degree $4$ over $\Q$. We show that $\rho_{HD_3,\ell}$ can be extended to $G_\Q$ for the case $(c,f) = (\frac15, \frac25)$; the same conclusion holds for $(c,f)=(1/10, 3/10)$ by a similar argument. The group Gal$(\Q(\zeta_5)/\Q)$ is generated by the automorphism 
		$\tau:\zeta_5\mapsto \zeta_5^3$. It maps  $\alpha_3(\frac25)=\{\f12,\frac 25,\frac35\}$ to $\{\f12,\frac15,\frac45\}$ and $ \beta_3(\frac15)=\{1,\frac{13}{10},\frac7{10}\}$ to $\{1,\frac9{10},\frac{11}{10}\}$. By Proposition \ref{prop: Kummer}, for any prime power $q$ coprime to $M$,
		$$\P\left(\alpha_3({f}),\beta_3({c});1; \F_q\right)={\phi_q(-1)}\P\left({\lc \frac12,\frac15,\frac45\rc,\lc 1,\frac9{10},\frac1{10}\rc};1; \F_q\right).$$In terms of representations, this means that $\rho_{HD_3,\ell}\cong  \chi_{-1}|_{G(M)}\otimes \rho_{HD_3,\ell}^\tau$,  where $\rho_{HD_3,\ell}^\tau$ is the conjugate of $\rho_{HD_3,\ell}$ by $\tau$. As shown above, $\rho_{HD_3,\ell}\otimes \chi_{-1}|_{G(M)} \cong \rho_{HD_3,\ell}$, therefore 
		$\rho_{HD_3,\ell}\cong \rho_{HD_3,\ell}^\tau$. 
		This proves that $\rho_{HD_3,\ell}$  can be extended to $G_\Q$, denoted by $\rho_{HD_3,\ell}^{BCM}$ by abuse of notation.

		For $(\frac15,\frac25)$, $\xi_1 \neq \xi_2$, which implies that  $\rho_{HD_3,\ell}^{BCM}$ is irreducible.  By Clifford theory, $\xi_1$ extends to an index-2 subgroup $G_K$ of $G_\Q$ so that  $\rho_{HD_3,\ell}^{BCM} = \text{Ind}_{G_K}^{G_\Q} \xi_1$. Since the map $\iota: \zeta_4\mapsto -\zeta_4$ in Gal$(\Q(\zeta_{20})/\Q(\zeta_{10}))$ swaps $\xi_1$ and  $\xi_2$, $K$ is not contained in $\Q(\zeta_{10})$. On the other hand, $\Q(\zeta_{20})$ contains three quadratic subfields, so $K$ is one of the two imaginary extensions outside $\Q(\zeta_{10})$. This shows that $\rho_{HD_3,\ell}^{BCM}$ is odd.
		
		For $(\frac1{10},\frac3{10})$, by formula \eqref{eq:3P2(1)} and simplification using the reflection and multiplication formulas of Gauss sums, we get $\xi_1 = \xi_2 = \(\frac {\sqrt 5}{\cdot}\)_2 \epsilon_\ell|_{G(20)}$, where $\(\frac {\sqrt 5}{\cdot}\)_2$ is a Hilbert quadratic norm residue symbol on $\Q(\zeta_{20})$. Clearly $\xi_1$ extends to $G_{\Q(\sqrt 5)}$ and not $G_\Q$. Hence $\rho_{HD_3,\ell}^{BCM}$ is irreducible and it equals  $\text{Ind}_{G_{\Q(\sqrt 5)}}^{G_\Q} \xi_1$. Moreover, viewed as an idele class character of $\Q(\sqrt 5)$ by class field theory, $\(\frac {\sqrt 5}{\cdot}\)_2$ is trivial at one real place of $\Q(\sqrt 5)$, and is the sign character at the other real place, resulting from the two real embeddings of $\sqrt 5$. This shows that $\rho_{HD_3,\ell}^{BCM}$ is odd.
		
		In conclusion, for each of the seven cases, $\rho_{HD_3,\ell}^{BCM}$ is odd, irreducible and, by Clifford theory, induced from a character of an index-2 subgroup of $G_\Q$. Hence it is modular by the work of Hecke.
		When $HD_3$ is defined over $\Q$, we compute the characteristic polynomials of each representation at suitable Frobenius conjugacy classes and identify the corresponding modular forms using the method described in \S \ref{ss:3.3}. For the other two cases, from \eqref{eq:3P2(1)} we identify the corresponding modular forms.  For example, for primes $p <200$ and $p\equiv 1 \mod 10$, $\P(HD_3(\frac15,\frac25);1;\F_p)=a_p(f_{20.3.d.a})=a_p(f_{100.3.b.c})$. Here $f_{20.3.d.a}$ is a cusp form of character $\chi_{-5}$ with coefficients in $\Z$ while the coefficients of $f_{100.3.b.c}$ are no longer in $\Z$.
	\end{proof}

	As explained before   Proposition \ref{prop:detrho3}, $ \rho_{HD_3,\ell}|_{G(N(c,f))}$ is a direct sum of two 1-dimensional representations and   
	$\det  \rho_{HD_3,\ell}(\Frob_\wp)= \psi_{c,f}(\Frob_\wp) N(\wp)^2$ at prime ideals $\wp$ of $\Z[\zeta_{M(c,f)}, 1/M(c,f)\ell]$. 
	Consequently,
	\begin{equation}\label{eq:tr-Alt-1}
	\Tr \, \text{Alt}^2 \rho_{HD_3,\ell}(\Frob_{\wp})=\det \rho_{HD_3,\ell}(\Frob_\wp)= \psi_{c,f}(\Frob_\wp) N(\wp)^2;
	\end{equation}
	\begin{equation}\label{eq:tr-Sym-1}
	\Tr \, \text{Sym}^2 \rho_{HD_3,\ell}(\Frob_{\wp}) +\det \rho_{HD_3,\ell}(\Frob_\wp)
	=\P (HD_3;\kappa_\wp)^2, 
	\end{equation}where Alt$^2\rho$ and Sym$^2\rho$ stand for the alternating square and symmetric square of a representation $\rho$ respectively.
	
	\medskip
	
	Next we consider the case when $\l\neq 1$.

	\begin{prop}\label{prop:main}Let $\ell$ be a prime and  $c,f\in \Q^\times$ be such that $\alpha_3(f)=\{\f12,f,1-f\},\beta_3(c)=\{1,\frac32-c,\f12+c\}$ form a primitive pair. Let $N(c,f)=lcd(\alpha_2(c,f)\cup \beta_2(c))$ as before. Then for any $\l\in \Z[\zeta_{M(c,f)}]\smallsetminus\{0,1\}$ 
		the 3-dimensional $\ell$-adic Galois representation $ \rho_{{\{\alpha_3(f), \beta_3(c);\l\}}, \ell}|_{G(N(c,f))}$ as in \eqref{eq:rho} possesses the following properties  
		at a  prime ideal $\wp$ of $\Z[\zeta_{N(c,f)}, 1/N(c,f)\ell]$: 
		\begin{equation}\label{eq:tr-Alt}
		\Tr \, \text{Alt}^2 \rho_{\{\alpha_3(f),\beta_3(c);\l\},\ell}|_{G(N(c,f))}(\Frob_{\wp})= \phi_\wp(1-1/\l) N(\wp)\cdot \P (\alpha_3(f),\beta_3(c);1/\l;\kappa_\wp)
		\end{equation}
		and
		\begin{multline}
		\label{eq:tr-Sym}
		\Tr \, \text{Sym}^2 \rho_{HD_3,\ell}|_{G(N(c,f))}(\Frob_{\wp})= \\ \P (\alpha_3(f),\beta_3(c);1/\l;\kappa_\wp)^2- \phi_\wp(1-1/\l) N(\wp)\cdot \P (\alpha_3(f),\beta_3(c);1/\l;\kappa_\wp),
		\end{multline}where $\phi_\wp(\cdot)$ is the quadratic character of $\kappa_\wp$.
		Further, $\text{Sym}^2 \rho_{\{\alpha_3(f),\beta_3(c);\l\},\ell}|_{G(N(c,f))}$ contains a $1$-dimensional sub-representation isomorphic to $\epsilon_\ell^2|_{G(N(c,f))}$.
	\end{prop}
	
	\begin{proof}
		Write $\rho$ for the $2$-dimensional $\ell$-adic representation $\rho_{\{\alpha_2(c,f), \beta_2(c); \l\}, \ell}$ of $G(N(c,f))$ and by $\ol{\rho}$ the representation $\rho_{\{\ol{\alpha_2}(c,f),\ol{\beta_2}(c) ; \l\}, \ell}$ of the same group, where $\ol{\alpha_2}(c,f) =  1-\alpha_2(c,f)$ and $\ol{\beta_2}(c) = \{1, \frac12 + c\}$. Then at a prime ideal $\wp$ of $\Z[\zeta_{N(c,f)},1/N(c,f) \ell]$ such that $\ord_\wp \l = 0$, we have
		$$\Tr \rho(\Frob_\wp) = -\omega^{(q-1){\frac{1+2f-2c}4}}(
		-1)\P(\alpha_2(c,f), \beta_2(c); 1/\l; \kappa_\wp)$$ and a similar conclusion for $\ol{\rho}$. 
		The eigenvalues $u_\wp, v_\wp$ of $\rho(\Frob_\wp)$ have absolute value $N(\wp)^{1/2}$, and those of $\ol{ \rho}(\Frob_\wp)$ are $N(\wp)/u_\wp$ and $N(\wp)/v_\wp$. Thus the $4$ eigenvalues of $\rho\otimes \ol{\rho}\,(\Frob_\wp)$ are $N(\wp)u_\wp/v_\wp,$ $N(\wp)v_\wp/u_\wp,$  $N(\wp), N(\wp)$. Write $\det \rho = \chi_\rho \epsilon_\ell$ for a finite order character $\chi_\rho$ of $G(N(c,f))$. Then $u_\wp v_\wp = \chi_\rho(\Frob_\wp)N(\wp)$ implies $N(\wp)/u_\wp = \chi_\rho(\Frob_p)^{-1}v_\wp$ and $N(\wp)/v_\wp = \chi_\rho(\Frob_p)^{-1}u_\wp$. Since the eigenvalues of Sym$^2 \rho(\Frob_\wp)$ are $u_\wp^2, v_\wp^2, u_\wp v_\wp$, we get
		$$ \chi_\rho^{-1}\,(\Frob_\wp)\Tr\, \text{Sym}^2\rho(\Frob_\wp) = \Tr\, \rho\otimes \ol{\rho}(\Frob_\wp) - N(\wp).$$
		Theorem \ref{thm:E-G} then implies that 
		$$\Tr \rho_{\{\alpha_3(f),\beta_3(c);\l\},\ell}|_{G(N(c,f))}(\Frob_\wp) = \phi_\wp(1-1/\l)\chi_\rho^{-1}(\Frob_\wp)\Tr\, \text{Sym}^2\rho(\Frob_\wp).$$ In other words, as representations of $G(N(c,f))$, $\rho_{\{\alpha_3(f),\beta_3(c);\l\},\ell}|_{G(N(c,f))}$ is isomorphic to $\chi_\rho^{-1}\otimes \text{Sym}^2\rho$ twisted by the Hilbert quadratic norm residue symbol $\(\frac{1-1/\l}{\cdot}\)_2$ on $\Q(\zeta_{N(c,f)})$. Hence 
		the eigenvalues of $\rho_{\{\alpha_3(f),\beta_3(c),\l\},\ell}|_{G(N(c,f))}(\Frob_\wp)$ are  $$c\cdot N(\wp)u_\wp/v_\wp,\quad c\cdot N(\wp)v_\wp/u_\wp,\quad c\cdot N(\wp), \quad \text{where } \, c=\phi_\wp(1-1/\l).$$ Therefore, 
		
		$$\text{Sym}^2\rho_{\{\alpha_3(f),\beta_3(c);\l\},\ell}|_{G(N(c,f))} \simeq \text{Sym}^2(\chi_\rho^{-1}\otimes\text{Sym}^2 \rho)\simeq \chi_\rho^{-2}\otimes \text{Sym}^2(\text{Sym}^2 \rho)$$
		and
		$$\text{Alt}^2\rho_{\{\alpha_3(f),\beta_3(c);\l\},\ell}|_{G(N(c,f)} \simeq  \text{Alt}^2(\chi_\rho^{-1}\otimes\text{Sym}^2 \rho).$$
		The last relation implies {that} the eigenvalues of $\text{Alt}^2\rho_{\{\alpha_3(c),\beta_3(f);\l\},\ell}|_{G(N(c,f)}(\Frob_\wp)$ are $$N(\wp)^2, \quad N(\wp)^2u_\wp/v_\wp, \quad \text{and } N(\wp)^2v_\wp/u_\wp.$$ This proves \eqref{eq:tr-Alt}. The equality \eqref{eq:tr-Sym} follows from the general identity for a representation $\rho'$ (see \cite[(67.8)]{CR}):  $$\Tr\text{Alt}^2 \rho'+\Tr\text{Sym}^2\rho'=\Tr\rho'\otimes \rho' =(\Tr\rho')^2.$$
		
		For the last assertion, note the following relation for 2-dimensional semisimple representations $\rho$ which can be proved by comparing the characters of both sides : 
		$$ \text{Sym}^2(\text{Sym}^2 \rho) = \text{Sym}^4\rho \oplus (\det\rho)^2.$$
		It follows that the degree-6 representation $\text{Sym}^2 \rho_{\{\alpha_3(f),\beta_3(c);\l\},\ell}|_{G(N(c,f))}$ decomposes into the sum of a degree-5 and a degree-1 sub-representations:
		\begin{eqnarray}\label{eq:sym4}
		\text{Sym}^2\rho_{\{\alpha_3(f),\beta_3(c);\l\},\ell}|_{G(N(c,f))} \simeq (\chi_\rho^{-2}\otimes \text{Sym}^4 \rho) \oplus \chi_{\rho}^{-2}(\det \rho)^2.
		\end{eqnarray}
		Since $\det \rho = \chi_\rho \epsilon_\ell|_{G(N(c,f))}$, the degree-1 sub-representation of $\text{Sym}^2 \rho_{\{\alpha_3(c),\beta_3(f);\l\},\ell}|_{G(N(c,f))}$ is $\epsilon_\ell^2|_{G(N(c,f))}$.
	\end{proof}

	\begin{proof}[Proof of Theorems \ref{thm:LLT2main} and \ref{thm:main}]Below we fix a pair $(c,f)$  such that $HD_1(c,f)$ and $HD_2(c,f)$ are primitive 
		and drop the reference to $c, f$ when there is no danger of confusion. Fix a prime $\ell$. 
		Using Lemma \ref{lem:6P5(1)}, we have, at a prime ideal $\wp$ of $\Z[\zeta_M, 1/M\ell]$,
		\begin{equation}\label{eq:P(HD1)}
		\P(HD_1;\kappa_\wp)=\sum_{t\in \kappa_\wp} \phi_\wp(t)\cdot  
		\,\P\left(\alpha_3(f),\beta_3(c);t;\kappa_\wp\right)^2.
		\end{equation}   
		
		Recall  that finite dimensional semisimple   representations {over fields of characteristic zero} of a group of any size are determined by their traces (see \cite[(27.13)]{CR} by Curtis and Reiner). We will use \eqref{eq:P(HD1)} to obtain an isomorphism between two Galois representations of $G(N)$ with different origins. For the left hand side, recall from Theorem \ref{thm:Katz} that $\Tr \rho_{HD_1,\ell}(\Frob_\wp)=-\phi_\wp(-1) \P(HD_1;\kappa_\wp)$. 
		
		For the right hand side, by Katz \cite{Katz,Katz09} there exists an $\ell$-adic sheaf $\mathcal H$ of \bk the affine line \bk $\mathbb A$  over $R=\Z[\zeta_N,1/\ell N]$ such that the action of $G(N)$ on the stalk of $\mathcal H_\l$ at each nonzero $\l \in R$ is a generically 3-dimensional representation,  
		characterized by the following trace formula: At any prime ideal $\wp$ of $ R$, and $\l \in R$ with $\ord_\wp \l=0$
		$$\text{Tr}\, \Frob_\wp(\mathcal H_\l)=\phi_\wp(-1)\P\left(\alpha_3(f),\beta_3(c);1/\l;\kappa_\wp\right).$$ Here and thereafter, to ease our notation, we sometimes identify a representation $\rho'$ with the underlying space $V$ on which it acts, and use $\Frob_\wp(V)$ to denote $\rho'(\Frob_\wp)$.
		By \cite[(30.9)]{CR},  $\text{Tr}\, \Frob_\wp(\mathcal H_\l\otimes \mathcal H_\l)=\P\left(\alpha_3(f),\beta_3(c);1/\l;\kappa_\wp\right)^2.$  
		Now consider the representation $\tilde \rho$ obtained from  $G(N)$ acting on  $V:=H^1_{\acute{e}t}(\mathbb A\otimes_R \overline \Q,\mathcal L_\phi \otimes  \mathcal H^{\otimes 2})$ where $\mathcal L_\phi$ is a rank-1 sheaf on $\mathbb A$ corresponding to the quadratic character $\phi$. At a prime ideal $\wp$ where the representation is unramified, the trace of $\tilde{\rho}(\Frob_\wp)$ on $V$ can be computed from the action of the Frobenius element in Gal$(\ol{\kappa_\wp}/\kappa_\wp)$ on 
		$H^1_{\acute{e}t}(\mathbb A\otimes_{\kappa_\wp} \ol{\kappa_\wp}, \mathcal L_\phi \otimes  \mathcal H^{\otimes 2})$ since $H^i_{\acute{e}t}(\mathbb A\otimes_{\kappa_\wp} \ol{\kappa_\wp}, \mathcal L_\phi \otimes  \mathcal H^{\otimes 2})=0 $ for $i\neq 1$. Expressed in terms of the action of the Frobenius $F_x$ at each point $x \in \kappa_\wp$, the Leftschetz formula gives $\Tr \tilde \rho (\Frob_\wp)=- \sum _{x\in \kappa_\wp^\times} \Tr(x)$ where $$\Tr(x)=\Tr F_x((\mathcal L_\phi \otimes  \mathcal H^{\otimes 2})_x) =\phi_\wp(x) \P(\alpha_3(f),\beta_3(c);1/x;\kappa_\wp)^2.$$ Putting all these together we get from \eqref{eq:P(HD1)} that, up to semisimplification, $(\chi_{-1}\otimes \rho_{HD_1,\ell})|_{G(N)}\cong \tilde \rho.$

		Next we will use this isomorphism to obtain a decomposition of $\rho_{HD_1,\ell}^{ss}|_{G(N)}$. Note that $\mathcal H^{\otimes 2}=\mathcal H\otimes \mathcal H$ decomposes naturally into $\text{Sym}^2\mathcal H\oplus \text{Alt}^2\mathcal H$ as $G(N)$-invariant sheaves. From the discussion of Proposition \ref{prop:main}, we further know that $\text{Sym}^2\mathcal H=\text{Sym}^4\rho \oplus (\det\rho)^2$, where $\rho$ stands for the $\ell$-adic sheaf on $\mathbb A$   such that at each $t \in \mathbb A$, $t \ne 0$, the stalk 
		at $t$ gives rise to the representation in Theorem \ref{thm:Katz} associated to the pair $\alpha_2(c,f), \beta_2(c)$.  
		Accordingly $H^1(\mathbb A\otimes_R \overline \Q,\mathcal L_\phi \otimes  \mathcal H^{\otimes 2})$ decomposes into  $H^1(\mathbb A\otimes_R \overline \Q,\mathcal L_\phi \otimes \text{Alt}^2 \mathcal H)\oplus H^1(\mathbb A\otimes_R \overline \Q,\mathcal L_\phi \otimes (\det\rho)^2)  \oplus H^1(\mathbb A\otimes_R \overline \Q,\mathcal L_\phi \otimes \text{Sym}^4 \rho)$ as $G(N)$-modules, and we write the corresponding decomposition of representations of $G(N)$ on these cohomology spaces in order as 
		$$\tilde \rho = \tilde {\rho}_{AS} \oplus \tilde {\rho}_{ES} \oplus \tilde {\rho}_{SS}.$$  
		This naturally yields the decomposition of the trace at each prime ideal $\wp$ of $R$:  $$\Tr \tilde {\rho}(\Frob_\wp) = -\P(HD_1,\wp)=AS_{c,f}(\wp)+ES_{c,f}(\wp)+SS_{c,f}(\wp)$$ where
		\begin{eqnarray*}
			AS_{c,f}(\wp)&:=&\Tr\,{\tilde{\rho}_{AS}}(\Frob_\wp)\\
			&=& -\Tr \, F_1((\text{Alt}^2 \mathcal H)_1) -\sum_{t\in \kappa_\wp ^\times,t\neq 1}\phi_\wp(t) \Tr\, F_t ((\text{Alt}^2\, \mathcal H)_t)\\
			&\overset{\eqref{eq:tr-Alt-1},\eqref{eq:tr-Alt}} =&-\psi_{c,f}(\Frob_\wp)N(\wp)^2-\sum_{t}\phi_\wp(t) \phi_\wp(1-1/t) N(\wp)\cdot \P(\alpha_3(f),\beta_3(c);1/t;\kappa_\wp)\\&\overset{\eqref{eq:P-by-induction}}=-&\psi_{c,f}(\Frob_\wp)N(\wp)^2-N(\wp)\cdot \P\left(\left\{\f12,\f12,f,1-f\right\},\left\{1,1,\f32-c,\f12+c\right\};1;\kappa_\wp\right)\\
			&\overset{Prop. \ref{prop:detrho3}}=&-N(\wp)^2-N(\wp)\cdot\P(HD_2;\kappa_\wp)\\
			&=& -N(\wp)^2+\phi_\wp(-1)N(\wp) \Tr \, \rho_{HD_2,\ell}(\Frob_\wp);\\
			ES_{c,f}(\wp)&:=& \Tr\,{\tilde{\rho}_{ES}}(\Frob_\wp) 
			=- \sum_{t\neq 0,1} \phi_\wp(t)N(\wp)^2 =  N(\wp)^2;   \\
			SS_{c,f}(\wp)&:=&\Tr\, {\tilde{\rho}_{SS}}(\Frob_\wp). 
		\end{eqnarray*}

		From the above expressions, we see that the representation of $G(N)$ on $H^1(\mathbb A\otimes_R \overline \Q,\mathcal L_\phi \otimes \text{Alt}^2 \mathcal H)\oplus H^1(\mathbb A\otimes_R \overline \Q,\mathcal L_\phi \otimes (\det\rho)^2) $ is isomorphic to $\epsilon_\ell\otimes\chi_{-1}\otimes \rho_{HD_2,\ell}|_{G(N)}$. 
		Consequently, we obtain the following isomorphism:  
		\begin{eqnarray}\label{eq:HD1=2+3}
		{\rho_{HD_1, \ell}}^{ss}|_{G(N)} \simeq (\epsilon_\ell\otimes \rho_{HD_2, \ell})|_{G(N)} \oplus (\tilde {\rho}_{SS} \otimes \chi_{-1})|_{G(N)},
		\end{eqnarray}
		showing the decomposition of ${\rho_{HD_1, \ell}}^{ss}|_{G(N)}$ into a sum of a degree-3 and a degree-2 representations. This proves Theorem \ref{thm:LLT2main}.

		When the pair $(c,f)$ is  among those  listed in (\ref{eq:7pair}), explicit computations as explained in \S \ref{ss:3.3} show that the representation $\rho_{HD_1,\ell}|_{G(N)}$ contains no irreducible subrepresentations of multiplicity larger than 1. Hence $\rho_{HD_1,\ell}$ and $\rho_{HD_2,\ell}$ are semi-simple and (\ref{eq:HD1=2+3}) holds without semi-simplification. This is Theorem \ref{thm:main}. 
	\end{proof}
	\medskip
	
	By Theorem \ref{thm:Katz}, iii), we can decompose the degree-$3$ representation $\rho_{HD_2,\ell}$ as $\rho_{HD_2,\ell}^{prim} \oplus \rho_{HD_2,\ell}^1$, and similarly, $\rho_{HD_1,\ell} =  \rho_{HD_1,\ell}^{prim} \oplus \rho_{HD_1,\ell}^1$. Contained in the proof above is the information on $\rho_{HD_2,\ell}^1$ and $\rho_{HD_1,\ell}^1$.
	
	\begin{cor}\label{cor:rho_{HD2}^1} For any pair $(c,f)$ in \eqref{eq:7pair} and any prime $\ell$, $$\rho_{HD_2(c,f),\ell}^1|_{G(N(c,f))} =  \chi_{-1}\epsilon_\ell |_{G(N(c,f))} \quad  and \quad \rho_{HD_1(c,f),\ell}^1|_{G(N(c,f))} =  \chi_{-1}\epsilon_\ell^2 |_{G(N(c,f))}.$$
	\end{cor}

	As a corollary, we can write down the following finite field analogue of \eqref{eq:Whipple2}.
	
	\begin{cor} For any pair $(c,f)$ in \eqref{eq:7pair}
		and any prime ideal $\wp$ of $\Z[\zeta_{N(c,f)},1/N(c,f)]$ we have 
		\begin{equation}
		-\P(HD_1;\kappa_\wp)=-N(\wp)\cdot \P(HD_2;\kappa_\wp)+SS_{c,f}(\wp).
		\end{equation}
	\end{cor}

	\subsection{A proof of Theorem \ref{thm:HD2}}\label{HD2} 
	For any $(c,f)$ as in \eqref{eq:7pair}, $HD_2(c,f)$ and $HD_2(f,c)$ may not agree by definition. Nonetheless the representations $\rho_{HD_2(c,f),\ell}$ and $\rho_{HD_2(f,c),\ell}$ are isomorphic because of the following Lemma.

	\begin{lemma}\label{lem:unorder}Given any $(c,f)$ in \eqref{eq:7pair}, for any prime power $q\equiv 1 \mod M(c,f)$,
		\begin{equation*}
		\P(\alpha_4(f),  \beta_4(c) ;1;\F_q)= \P (  \alpha_4(c),  \beta_4(f);1;\F_q).
		\end{equation*}
	\end{lemma} 
	\begin{proof}
		
		$$\P(\alpha_4(f), \beta_4(c) ;1;\F_q)=\pPPq43{\phi&\phi&\omega^{(q-1)f}&\omega^{(q-1)(1-f)}}{&\eps& \omega^{(q-1)(\f12-c)}& \omega^{(q-1)(\f12+c)}}{1;q}$$
		
		By Proposition \ref{prop: Kummer}, the right hand side coincides with $ \pPPq43{\phi&\phi&\omega^{(q-1)c}&\omega^{(q-1)(1-c)}}{&\eps& \omega^{(q-1)(\f12-f)}& \omega^{(q-1)(\f12+f)}}{1;q}$ which is $\P (  \alpha_4(c),  \beta_4(f);1;\F_q)$.
	\end{proof}
	
	\medskip
	
	It is clear that when $(c,f)$ is one of the first five pairs, $HD_2$ is also defined over $\Q$. When $(c,f)$ is either $\(\frac 15,\frac25\)$ or $\left(\frac 1{10},\frac3{10}\right)$, one has

	\begin{lemma}\label{lem:HD-Q}
		Let $(c,f)$ be either $\(\frac 15,\frac25\)$ or $\left(\frac 1{10},\frac3{10}\right)$. Then  the  representation $\rho_{HD_2,\ell}$ attached to the hypergeometric datum   $HD_2=\{\alpha_4(f),\beta_4(c);1\}$   can be extended to $G_\Q$.
	\end{lemma}
	\begin{proof}
		When $(c,f)=(\frac15,\frac25)$, $M = 10$ and $\rho_{HD_2, \ell}$ is a representation of $G(10) = G(5)$. Since $HD_2$ is self-dual and defined over $K=\Q(\zeta_5+\zeta_5^{-1})=\Q(\sqrt 5)$, $\rho_{HD_2, \ell}$ is invariant under the Galois conjugation sending $\zeta_5$ to $\zeta_5^{-1}$ so that it can be extended to $G_K$, the absolute Galois group of $K$. 
		The Galois group $\text{Gal}(K/\Q)$ is generated by $\tau: \zeta_5\mapsto \zeta_5^3$, which maps $\alpha_4(f)=\{\f12,\f12,\frac25,\frac35\}$ to $\{\f12,\f12,\frac15,\frac45\}$ and  $\beta_4(c)=\{1,1,\frac3{10},\frac7{10}\}$ to $\{1,1,\frac9{10},\frac1{10}\}$. For all primes $p \equiv 1 \mod 5$, we have $p \equiv  1 \mod 10$ and by  Lemma \ref{lem:unorder}, $$\P(\alpha_4(f),\beta_4(c);1;\F_p)=\P(\alpha_4(c),\beta_4(f);1;\F_p)=\P\left(\lc \f12,\f12,\frac15,\frac45\rc,\lc 1,1,\frac9{10},\frac1{10}\rc;1;\F_p\right).$$ This shows that $\rho_{HD_2, \ell}$ of $G_K$ is invariant under conjugation by $\tau$, and hence it can be extended to $G_{\Q}$. The case of $(c,f)=\left(\frac 1{10},\frac3{10}\right)$  can be handled in a parallel manner.
	\end{proof}

	Now that we know each $\rho_{HD_2,\ell}$ has an extension to $G_\Q$, we proceed to prove the assertions about $\rho_{HD_2,\ell}^{BCM}$ in Theorem \ref{thm:HD2}. For $(c,f)=(\f12,\f12),(\f12,\frac13),(\f12,\frac16)$, this has been established in \cite{LTYZ}. We prove the remaining cases. 
	If $HD_2$ is defined over $\Q$, from Theorems \ref{thm:Katz} and \ref{thm:KBCM}, we know each $\rho_{HD_2,\ell}^{BCM}\cong\rho_{HD_2,\ell}^{BCM,prim}\oplus \rho_{HD_2,\ell}^{BCM,1}$ where $\rho_{HD_2,\ell}^{BCM,1}$ is 1-dimensional. When $HD_2$ is not defined over $\Q$, we continue to use  $\rho_{HD_2,\ell}^{BCM,prim}$ and $\rho_{HD_2,\ell}^{BCM,1}$ to denote an extension of the degree-2 and degree-1 pieces of $\rho_{HD_2,\ell}$ to $G_\Q$ with integral traces. 
	
	We first deal with $\rho_{HD_2,\ell}^{BCM,prim}$ by using Theorem \ref{thm:modulairty} to prove the modularity of  $\rho_{HD_2,\ell}^{BCM,prim}[t]$ where $t$ is computed by \eqref{eq:t} and identify the corresponding cuspidal Hecke eigenform. Then we determine $\rho_{HD_2,\ell}^{BCM,1}$.
	
	Using the \texttt{Magma} program called ``Hypergeometric Motives over $\Q$" implemented by Watkins \cite{Watkins-HGM-documentation},
	the characteristic polynomial of $\rho_{HD_2,\ell}^{BCM,prim}{[t]}(\text{Fr}_p)$ for $p\ge 5$  
	can be computed efficiently {by choosing $\ell \ne p$} as mentioned in \S\ref{ss:3.3}. Thus for each $p>3$, the trace and determinant of $\rho_{HD_2,\ell}^{BCM,prim}{[t]}(\text{Fr}_p)$ can be read off.  When  $HD_2$ is defined over $\Q$, the characteristic polynomials of $\rho_{HD_2,\ell}^{BCM,prim}{[t]}(\text{Fr}_p)$  all  have  coefficients in $\Z$. The representation $\rho_{HD_2,\ell}^{BCM,prim}$ is unramified outside $\{2,3, \ell\}$ for $(c,f)=(\frac13,\frac13)$ and $(\frac16,\frac16)$, 
	and these $\ell$-adic representations are compatible as $\ell$ varies. Hence   $\rho_{HD_2,\ell}^{BCM,prim}$ and its determinant  can ramify only at $2$ and $3$. 
	To compute it, we may choose $\ell=2$ or 3. 
	As $n=4$ and dimension of $\rho_{HD_2,\ell}^{BCM,prim}$ is 2, by iii) of Theorem \ref{thm:KBCM}, and the absolute value of the roots of $\rho_{HD_2,\ell}^{BCM,prim}(\Frob_p)$ are $p^{3/2}$ for $p>3$, one sees that $\det \rho_{HD_2,\ell}^{BCM,prim}=\chi_{c,f}\epsilon_\ell^3$ where {$\epsilon_\ell$ denotes the unramified  $\ell$-adic cyclotomic character with $\epsilon_\ell(\Frob_p) = p$} and $\chi_{c,f}$ has order at most 2  because of integer coefficients.  Since $\chi_{c,f}$ is unramified outside $2$ and $3$, its conductor  divides 24. From computing the first few characteristic polynomials of $\rho_{HD_2,\ell}^{BCM,prim}{[t]}(\text{Fr}_p)$, we confirm that $\chi_{c,f}$ is the trivial character for each case. So $\det \rho_{HD_2,\ell}^{BCM,prim}=\epsilon_\ell^3$ which implies $\rho_{HD_2,\ell}^{BCM,prim}$ and hence each
	$\rho_{HD_2,\ell}^{BCM,prim}[t]$ {(which is $2$-dimensional) is odd for one and hence all $\ell$}. The irreducibility of $\rho_{HD_2,\ell}^{BCM,prim}$ follows from the lemma below since it has degree $2$ and determinant $\epsilon_\ell^3$.

	\begin{lemma} \label{irreducible} Suppose the hypergeometric datum $HD =\{\alpha,\beta; 1\}$ in Theorem \ref{thm:Katz} is such that, for all primes $\ell$, the representation $\rho_{HD,\ell}^{prim}$ of $G(M)$  has a degree-$2$ subrepresentation which can be extended to a representation $\rho_\ell$ of $G_\Q$ and the $\rho_\ell$ form a compatible family. If $\det \rho_\ell$, up to a finite twist, is $\epsilon_\ell^m$ for an odd natural number $m$, then $\rho_\ell$ is irreducible.
	\end{lemma}

	\begin{proof} Suppose otherwise. Then the semisimplification of $\rho_\ell$ 
		is  of the form $\chi_1 \epsilon_\ell^a \oplus \chi_2 \epsilon_\ell^b$,  where $a, b$ are non-negative integers with $a+b=m$ and $\chi_1, \chi_2$ are characters of $G_\Q$ of finite order.  Thus the eigenvalues of $\rho_{HD,\ell}^{prim}(\Frob_\wp)$ would not have the same  absolute values at almost all prime ideals $\wp$, contradicting Theorem \ref{thm:Katz}, iii). 
	\end{proof}

	To apply Theorem \ref{thm:modulairty} to conclude the modularity of $\rho_{HD_2,\ell}^{BCM,prim}[t]$, choose $\ell>5$ to be a prime which splits completely in $\Q(\zeta_M)$. Then the action of $\rho_{HD_2,\ell}^{BCM,prim}[t]$ restricted to a decomposition group $D_\ell$ at $\ell$ is isomorphic to the action of $\rho_{HD_2, \ell}^{prim}[t]$ restricted to a decomposition group $D_\wp$ at a prime $\wp$ of $\Q(\zeta_M)$ above $\ell$, which is known to be crystalline by Katz \cite{Katz}.

	For $(c,f)=(\frac 13,\frac13)$, the graph of $e_{\alpha, \beta}$ with $\alpha=\{\f12,\f12,\frac13,\frac23\}, \beta=\{1,1,\frac76,\frac56\}$ is as follows. 
	\begin{center}
		\includegraphics[scale=0.3,origin=c]{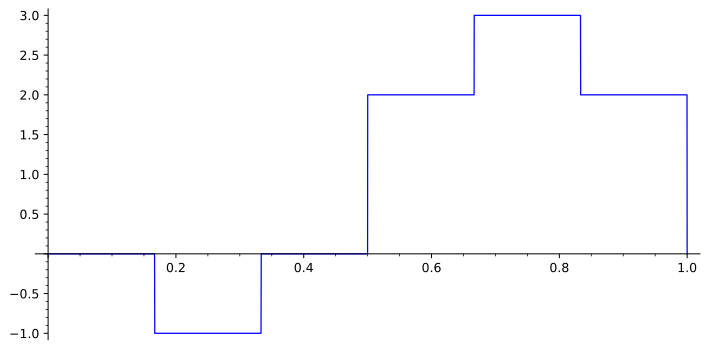}
		
	\end{center}In this case $\min e_{\alpha,\beta}=-1$, $t=1-\frac{4-2}2=0$, and $w(HD_2)=3-(-1)=4.$  This means $\rho_{HD_2,\ell}^{BCM,prim}$ has Hodge-Tate weights $\{0, 3\}$. By Theorem \ref{thm:modulairty}, $\rho_{HD_2,\ell}^{BCM,prim}$ is modular and the corresponding modular form is of weight-4 with integer coefficients and by Theorem \ref{thm:Serre} its level divides  $2^8\cdot 3^5$. From the list of cuspidal Hecke eigenforms of weight-4 and level dividing  $2^8\cdot 3^5$ listed in \href{https://sites.google.com/view/ft-tu/research/database?authuser=0}{\uline{the database}}. We see that only $f_{6.4.a.a}$  
	has the matching $p$-coefficients for all primes within the range $5\le p<100$.
	
	When $(c,f)=(\frac 16,\frac 16)$, $\min =0, \, \max=2$, $w(HD_2)=2$ and $t=-1$. By a similar analysis we know $\rho_{HD_2,\ell}^{BCM,prim}[-1]$ is modular and is isomorphic to the $\ell$-adic Galois representation attached to  $f_{24.2.a.a}$.
	
	Next we determine $\rho_{HD_2,\ell}^{BCM,1}$. We implemented a \texttt{SageMath} program to compute $H_p(\alpha,\beta;1)$.  From  Theorem \ref{thm:KBCM},  we know $ \rho_{HD_2,\ell}^{BCM,1}$ is 1-dimensional and its value at $\Frob_p$ can be computed by  
	$$ \Tr\rho_{HD_2,\ell}^{BCM,1}(\Frob_p)=p^{{(4-m)/2}}H_p(\alpha,\beta;1)- \Tr\rho_{HD_2,\ell}^{BCM,prim}(\Frob_p),$$
	where $m$ is the number of elements in $\{1,1,3/2-c,1/2+c\}$ that are integers, in these cases $m=2$ or 4 and $n-m$ is even.
	So $\rho_{HD_2,\ell}^{BCM,1}(\Frob_p)=\varphi_{c,f}(p)p$ for some character $\varphi_{c,f}$ of order at most $2$ and conductor dividing 24. By computing the first few primes $p$, $\varphi_{c,f}$ is identified quickly. For example, when $\alpha=\{\f12,\f12,\frac13,\frac23\}$, $\beta=\{1,1,\frac76,\frac56\}$, for $p \in \{7,11,13,17,19\}$ the values of $\rho_{HD_2,\ell}^{BCM,1}(\Frob_p)$ computed from $pH_p(\alpha,\beta;1)-\Tr\rho_{HD_2,\ell}^{BCM,prim}(\Frob_p)$ are $7,-11,13,-17,19$, respectively, which agree with $\(\frac{-3}p\)p.$
	
	For the last two cases $(c,f) = \left(\frac 15,\frac25\right)$ and $\left(\frac 1{10},\frac3{10}\right)$, each $\rho_{HD_2, \ell} = \rho_{HD_2,\ell}^{prim} + \rho_{HD_2,\ell}^{1}$ is a representation of $G(10)$. Since $\rho_{HD_2,\ell}$ and $\rho_{HD_2,\ell}^{1}$ can be extended to $G_\Q$, so can $\rho_{HD_2,\ell}^{prim}$, and each extension  
	is unique up to twists by the $4$ characters of $G_\Q$ trivial on $G(10)$. Unlike the previous $5$ cases, there is no preferred choice; we shall choose extensions with integral traces. From iii) of Theorem \ref{thm:Katz}  and numeric computation we know \bk $\det\rho_{HD_2,\ell}^{prim}=(\epsilon_\ell^3)|_{G(10)}$. Thus any extension of $\rho_{HD_2,\ell}^{prim}$ from $G(10)$ to $G_\Q$ has determinant either $\epsilon_\ell^3$ or $\chi_5\otimes \epsilon_\ell^3$, and hence is odd. It is also irreducible by Lemma \ref{irreducible}. 
	By Corollary \ref{cor:rho_{HD2}^1},   
	$\rho_{HD_2,\ell}^{1}|_{G(20)} = \epsilon_\ell|_{G(20)}$. Thus we can obtain the  trace of $\rho_{HD_2,\ell}^{BCM,prim}$ at any $\Frob_p$ with  $p\equiv 1\mod 20$  from computing $\Tr \rho_{HD_2,\ell}(\Frob_\wp)$ at any prime ideal $\wp$ of $\Z[\zeta_{20}]$ above such $p$ and coprime to $\ell$ which, by Remark \ref{rk:density}, allows us to identify the corresponding modular form up to quartic twists, parallel to the first 5 cases. The degree-1 piece can be determined accordingly.
	
	\subsection{A proof of Theorem \ref{thm:2}} 
	For each pair $(c,f)$ listed in \eqref{eq:7pair}, we will proceed in a way similar to the proof of Theorem \ref{thm:HD2} by dealing with $\rho_{HD_1,\ell}^{BCM,prim}$ and then its $1$-dimensional complement $\rho_{HD_1,\ell}^{BCM,1}$ which  can be determined the same way as the end of the proof of Theorem \ref{thm:HD2}.

	\begin{prop}\label{cor:1}
		For each pair $(c,f)$ in \eqref{eq:7pair}, the 4-dimensional representation  $\rho_{HD_1,\ell}^{BCM,prim}$ is the sum of two odd degree-2 irreducible automorphic representations of $G_\Q$.
	\end{prop}
	\begin{proof} Fix a  pair $(c,f)$ in \eqref{eq:7pair}. Since  $HD_1$ is defined over $\Q$ and $HD_2$ can be extended to $\Q$,  the representations $\rho_{HD_1,\ell}|_{G(N)}$ and $(\epsilon_\ell\otimes \rho_{HD_2, \ell})|_{G(N)}$ can be extended to $G_\Q$ as $\rho_{HD_1,\ell}^{BCM}$ and $\epsilon_\ell\otimes \rho_{HD_2, \ell}^{BCM}$ respectively. Hence by (\ref{eq:HD1=2+3}), $(\tilde {\rho}_{SS} \otimes \chi_{-1})|_{G(N)}$ can be extended to $\tilde {\rho}_{SS}^{BCM} \otimes \chi_{-1}$ of $G_\Q$.  Therefore $\rho_{HD_1,\ell}^{BCM}$ decomposes as the sum of a finite twist of $\epsilon_\ell\otimes \rho_{HD_2, \ell}^{BCM}$ and a finite twist of $\tilde {\rho}_{SS}^{BCM} \otimes \chi_{-1}$, written as \begin{equation}\label{eq:prim-decomp}
		\rho_{HD_1,\ell}^{BCM,prim}\cong \sigma_{HD_1,alt,\ell} \oplus \sigma_{HD_1,sym,\ell}
		\end{equation} where $\sigma_{HD_1,sym,\ell}$ and $\sigma_{HD_1,alt,\ell}$ are both 2-dimensional. Here we use $sym$ and $alt$ in the subscripts in order to keep track of their origin. It is useful to keep in mind that  they will play asymmetric roles in general. From Theorem \ref{thm:HD2}, $\det\sigma_{HD_1,alt,\ell}=\det (\epsilon_\ell\otimes \rho_{HD_2,\ell}^{prim})=\epsilon_\ell^5$  or $\chi_5\epsilon_\ell^5$.  
		We also determine  $\det\rho_{HD_1,\ell}^{BCM,prim}=\epsilon_\ell^{10}$ the same way as the proof of Theorem \ref{thm:HD2}. Hence $\sigma_{HD_1,sym,\ell}$ and $\sigma_{HD_1,alt,\ell}$ have the same determinant and are both odd. They are irreducible by Lemma \ref{irreducible}.  
		Since $\rho_{HD_1,\ell}^{BCM,prim}$ and $\sigma_{HD_1,alt,\ell}$ are both crystalline for a large prime $\ell \equiv 1 \mod N$ by Katz \cite{Katz}, so is $\sigma_{HD_1,sym,\ell}$. By Theorem \ref{thm:modulairty}, both $\sigma_{HD_1,alt,\ell}$ and $\sigma_{HD_1,sym,\ell}$ are  modular.  
	\end{proof}

	Our next task is to determine the modular forms $\ff_{c,f}$ and $\fg_{c,f}$ corresponding to $\sigma_{HD_1(c,f),sym,\ell}$ and $\sigma_{HD_1(c,f),alt,\ell}$ respectively in the decomposition \eqref{eq:prim-decomp} of $\rho_{HD_1,\ell}^{BCM,prim}$.
	
	As each $HD_1$ is defined over $\Q$, using \text{Magma}, one can compute the factorization of the characteristic polynomial of $\rho_{HD_1,\ell}^{BCM,prim}[t](\text{Fr}_p)$ over $\Z$ for any unramified prime $p$, which 
	is always a product of two 
	degree-2 polynomials over $\Z$.  For example,
	\medskip
	
	\texttt{H1:=HypergeometricData([1/2,1/2,1/2,1/2,1/3,2/3],[1,1,1,1,1/6,5/6]);}
	
	\texttt{Factorization(EulerFactor(H1,1,7));}
	
	\medskip
	
	The output is
	
	$<16807*\$.1^2 - 56*\$.1 + 1,1>,$
	
	$<16807*\$.1^2 + 88*\$.1 + 1, 1>.$
	\medskip
	
	Our next task is to assign the degree-2 factors to  $\sigma_{HD_1,sym,\ell}$ and  $\sigma_{HD_1,alt,\ell}$ respectively. 
	By Theorem \ref{thm:main}, $\sigma_{HD_1,alt,\ell}|_{G(N)}$ and $\(\epsilon_\ell\otimes\rho_{HD_2,\ell}^{BCM,prim}\)|_{G(N)}$ agree. This implies that at each unramified prime $p$, $\Tr \sigma_{HD_1,alt,\ell}(\Frob_p)$ and $p\cdot \Tr  \rho_{HD_2,\ell}^{BCM,prim}(\Frob_p)$ differ by a root of unity. Since they both have integer coefficients, they differ at most by a quadratic character. 
	This allows us to recognize which quadratic factor is from  $\sigma_{HD_1,alt,\ell}|_{G(N)}$. For instance, 
	\medskip
	
	\texttt{H2:=HypergeometricData([1/2,1/2,1/3,2/3],[1,1,1,1]);}
	
	\texttt{Factorization(EulerFactor(H2,1,7));}
	
	\medskip
	
	The output is
	
	$ <343*\$.1^2 - 8*\$.1 + 1, 1> $
	
	\noindent from which we know the first quadratic factor computed from H1 is from $\sigma_{HD_1,alt,\ell}.$ The separation of the quadratic factors allows us to compute the trace and determinant of $\sigma_{HD_1,sym,\ell}$ and $\sigma_{HD_1,alt,\ell}$ respectively. Similar to  the proof of Theorem \ref{thm:HD2}, the determination of $\ff_{c,f}$ and $\fg_{c,f}$ boils down to the application of Theorem \ref{thm:Serre} as outlined in \S \ref{ss:3.3}.

	\section{Hypergeometric values and periods of modular forms}\label{sec:mf}
	
	In this section we study untruncated hypergeometric series. Our main purpose is to relate hypergeometric values and periods of modular forms.
	\subsection{Modular forms and differential equations}
	In \cite{Zagier-top-diff},  Zagier used
	$$\pFq43{\f12&\f12&\f12&\f12}{&1&1&1}{1}=\frac{1}{2\pi i}\oint_{|t|=1} \pFq21{\f12&\f12}{&1}{t}\pFq21{\f12&\f12}{&1}{1/t}\frac{dt}t$$ with $t=\l(\tau)$, the modular lambda function, to obtain
	\begin{equation*}
	\pFq43{\f12&\f12&\f12&\f12}{&1&1&1}{1}=\frac{16}{\pi^2}L(f_{8.4.a.a},2).
	\end{equation*}
	In \cite{Stiller84} Stiller studied periods of modular forms and their relation to the Eichler-Shimura parabolic cohomology theory. 
	It is known in the literature (see \cite{Stiller84} by Stiller and \cite[Theorem 1]{Yang04} by Yang) that
	\begin{theorem}
		Let $\G$ be a  Fuschian group commensurable with $SL_2(\Z)$ and let $t$ be a non-constant modular function of $\G$. Given a weight-$k$ meromorphic  modular form $f$ on $\G$, write $f(\tau)=F(t(\tau))$ locally. Then $F(t),\tau F(t),$ $\cdots, \tau^k F(t)$
		satisfy an order $k+1$ differential equation in variable $t$ which has only regular singularities.
	\end{theorem}Because of this reason the period formulas   below often involve polynomials in $\tau$.
	\medskip
	
	To prove Theorem \ref{thm:period}, we first recall  some relevant properties of hypergeometric functions \cite[Chapters 15, 16]{DLMF} (For more detail, see \cite{AAR, Slater}  for example.): 
	\begin{enumerate}
		\item[(P.1)] { \href{https://dlmf.nist.gov/16.8}{\cite[(16.8.7)]{DLMF}}} The  functions  
		$$
		(-z)^{-a_1}\pFq32{a_1&1+a_1-b_2&1+a_1-b_3}{&1+a_1-a_2&1+a_1-a_3}{\frac1z}
		~{\rm and}~ \pFq32{a_1&a_2&a_3}{&b_2&b_3}z$$  satisfy the same 3rd order differential equation.
		\item[(P.2)] \label{item: Clausen HDEs} { \cite[\href{https://dlmf.nist.gov/15.10}{\S 15.10}, \href{https://dlmf.nist.gov/16.8}{\S 16.8}]{DLMF}} The following four functions
		$$\pFq21{a&b}{&a+b+\frac12}{z}^2, \quad z^{\f12-a-b}  \pFq21{a&b}{&a+b+\f12}{z}\pFq21{\f12-a&\f12-b}{&\frac32-a-b}{z},$$ $$\pFq32{2a&2b&a+b}{&2a+2b&a+b+\frac 12}{z},\quad
		z^{\f12-a-b}\pFq32{\f12&a-b+\f12&b-a+\f12}{&a+b+\f12&\frac32-a-b}{z}
		$$
		satisfy the same 3rd order differential equation (ODE) 
		and hence are linearly dependent.
		
		\item[(P.3)] {\href{https://dlmf.nist.gov/16.12}{\cite[(16.12.2)]{DLMF}}}  The Clausen's formulas give two identities related to the functions in (P.2): 
		\begin{equation}\label{eq:Clausen}
		\pFq21{a&b}{&a+b+\frac12}{z}^2=\pFq32{2a&2b&a+b}{&2a+2b&a+b+\frac 12}{z}
		\end{equation}
		\begin{equation*}
		\pFq21{a&b}{&a+b+\f12}{z}\pFq21{\f12-a&\f12-b}{&\frac32-a-b}{z}=\pFq32{\f12&a-b+\f12&b-a+\f12}{&a+b+\f12&\frac32-a-b}{z}.
		\end{equation*}
	\end{enumerate}
	Moreover,  $ \pFq21{a&b}{&a+b+\f12}{z}$ and $z^{\f12-a-b}\pFq21{\f12-a&\f12-b}{&\frac32-a-b}{z}$ satisfy the same ordinary differential equation.

	\medskip

	From (P.1) above, we denote by $L$ the order-3 ordinary differential operator which annihilates  $
	z^{-1}\pFq32{\f12&a-b+\f12&b-a+\f12}{&a+b+\f12&\frac32-a-b}{1/z}$ and $\pFq32{\f12&a-b+\f12&b-a+\f12}{&a+b+\f12&\frac32-a-b}{z} $, then $L\circ L$ annihilates 
	both
	$$
	z^{-1}\pFq32{\f12&a-b+\f12&b-a+\f12}{&a+b+\f12&\frac32-a-b}{1/z}\pFq32{\f12&a-b+\f12&b-a+\f12}{&a+b+\f12&\frac32-a-b}{z} $$ and   $$z^{-1/2} \pFq32{\f12&a-b+\f12&b-a+\f12}{&a+b+\f12&\frac32-a-b}{z}^2 .$$  Letting $z=u^2$, we conclude that $$\pFq32{\f12&a-b+\f12&b-a+\f12}{&a+b+\f12&\frac32-a-b}{u^2}^2 \quad {\rm and}\quad \pFq21{a&b}{&a+b+\f12}{u^2}^4$$ satisfy the same differential equation.

	\medskip
	\subsection{The occurrence of $\ff_{c,f}$ in $w(HD_1)=6$ cases} We show how $\ff_{c,f}$ in Theorem \ref{thm:2} for the three cases where $w(HD_1)=6$ arises naturally in connection with a $_2F_1$ function.
	
	\begin{theorem}\label{thm:wt6} When $(c,f)=(\f12,\frac12), (\f12,\frac13),(\frac13,\frac13)$, the monodromy group $\G(c,f)$ of the order-2 ODE satisfied by $\pFq21{\frac{1+2f-2c}4&\frac{3-2c-2f}4}{&\frac32-c}{z^2}$ is a genus 0 arithmetic Fuchsian group. The space of weight-6 holomorphic modular forms  for $\G(c,f)$
		is $1$-dimensional, generated by a Hecke eigenform $\fh_{c,f}$ with eigenvalue $a_p(\ff_{c,f})$ for almost all primes $p>7$. 
	\end{theorem}

	\begin{proof}
		We first compute the monodromy group $G(c,f)$ of the 2nd order ODE satisfied by \begin{equation}\label{eq:F}
		{}_2F_1(x):=\pFq21{\frac{1+2f-2c}4&\frac{3-2c-2f}4}{&\frac32-c}{x}=\pFq21{a&b}{&a+b+\f12}{x},
		\end{equation} (namely letting $a=\frac{1+2f-2c}4, b=\frac{3-2c-2f}4$)  which is an arithmetic triangle group with angles $\pi|a+b-\f12|,\frac{\pi}2,\pi|a-b|$. For related background, see \cite[\S3.2.4]{Win3X}. 
		Then under the change of variable $x=z^2$, the singularities of the ODE in the variable $z$ satisfied by $ {}_2F_1(z^2)$ are $\{\infty,0,1,-1\}$ and the monodromy group is an index-2 subgroup $\G(c,f)$ of $G(c,f)$.
		We will give a description for $\G(c,f)$ as a  matrix group for each case, and then describe  the space $S_6(\G(c,f))$ of weight-6 holomorphic modular forms on $\G(c,f)$ which vanish at each cusp if there is any.  
		
		In what follows, we use the signature $(e_1,e_2,\cdots, e_r)$ to denote a genus-zero group $\G$ which has  $r$ in-equivalent elliptic points and cusps together of respective order $e_1$, $e_2$, $\cdots$, $e_r$, where $e_i\in \Z_{>1}\cup\{\infty\}$.  According to \cite[Theorem 2.23]{Shimura-introduction} by Shimura, when $k \ge 4$  is even, the dimension of the space of weight-$k$ holomorphic cusp forms on $\G$ is 
		
		$$
		\dim S_k(\G)=-k+1 + \sum_{i=1}^r  \lfloor \frac{k}2(1-1/e_i)\rfloor -\#\{cusps\}.
		$$ 
		It follows that $\dim S_6(\G(c,f))=1$ when $w(HD_1)=6$.

		For $(c,f)=(\f12,\f12)$, by \eqref{eq:F},  $_2F_1(x)=\pFq21{\frac14&\frac14}{&1}x$. In this case the signature of $G(\f12,\f12)$ is $(\infty,\infty,2)$, and $G(\f12,\f12)$ is isomorphic to $\G_0(2)$. The 
		group $\G(\f12,\f12)$  has signature  $ (\infty,\infty,2,2)$ and is isomorphic to the level 4 index-6 subgroup of $SL_2(\Z)$ labelled by \href{https://mathstats.uncg.edu/sites/pauli/congruence/csg0.html}{$4C^0$} in \cite{Congruence}. The group $4C^0$ is a supergroup of $\G_0(8)$. By the above dimension formula, $\dim S_{6}(\G(\f12,\f12))=-5+ 2\lfloor \frac 64 \rfloor+2(3-1)=1$. The space is generated by $f_{8.6.a.a}(\tau)$,  which is $\ff_{\f12,\f12}(\tau)$ in Theorem \ref{thm:2}. \bk
		
		For $(c,f)=(\f12,\frac13)$,  $_2F_1(x)=\pFq21{\frac16&\frac13}{&1}x$. In this case the signature of $G(\f12,\frac13)$ is $(\infty,6,2)$. The group $G(\f12,\frac13)$ can be realized as  
		$$
		\G_0^+(3)=\left\langle w_3:=\frac{1}{\sqrt 3}\M0{-1}30,T:=\M1101\right\rangle, 
		$$ and its index-2 subgroup  $\G(\f12,\frac13)$ has signature $(2,2,3,\infty)$. This group  and $\G_0(3)$ are generated by
		$$
		\left\{w_3, \, (w_3 T^{-1})^2=\M{-1}1{-3}2, T^2\right\}, \quad \mbox{and} \quad \left\{T,\, \M{-1}1{-3}2,\, \M{-1}03{-1} \right\},
		$$ respectively.  Their intersection  has signature \bk $(3,3,\infty,\infty)$, generated by
		$$
		\left\{ T^2, \,\M{-1}1{-3}2,\,\M10{-6}1, \,\M{-1}{-1}32=T^{-1}\M{-1}1{-3}2 T \right\}.
		$$This {intersection} group has a supergroup, which is the index-2 subgroup of $SL_2(\Z)$ with signature $(3,3,\infty)$. It has a weight-6 cuspform  $f_{4.6.a.a}(\tau)=\eta(2\tau)^{12}$,  which is $\ff_{\f12,\frac13}(\tau)$ in Theorem \ref{thm:2}. \bk So $S_6(\G(\f12,\frac13))=\langle f_{4.6.a.a}(\tau)\rangle.$

		For $(c,f)=(\frac13,\frac13)$,  $_2F_1(x)=\pFq21{\frac14&\frac5{12}}{&\frac 76}x$. Its monodromy group   has signature  $(6,6,2)$, and the  index-2 subgroup  $\G(\frac13,\frac13)$ has signature $(3,3,2,2)$, which can be realized as the subgroup $\G_0^6(1)$ of $PSL_2(\R)$ 
		arising from the norm $1$ group of the maximal order of the quaternion algebra over $\Q$ with discriminant $6$. By the Jacquet-Langlands correspondence \cite{JL},  $\fh_{c,f}$  the generator of $S_6(\G(\frac13,\frac13))$ and  the Hecke eigenform $f_{6.6.a.a}(\tau)$,  which is $\ff_{\frac13,\frac13}(\tau)$ in Theorem \ref{thm:2}, \bk are both Hecke eigenforms with the same eigenvalue at each prime $p>3$.
	\end{proof}
	
	\subsection{The case $(c,f)=(\f12,\f12)$ - a proof of Theorem \ref{thm:period}} The goal of this subsection is to  prove Theorem \ref{thm:period}.  
	To continue, we first recall some known results below. The Jacobi theta functions
	\begin{eqnarray*}
		\theta_2(\tau)&=&\displaystyle \sum_{k\in \Z} q^{(2k+1)^2/8}=2\frac{\eta(2\tau)^2}{\eta(\tau)},\\ \theta_3(\tau)&=&\sum_{k\in\Z}q^{k^2/2}=\frac{\eta(\tau)^5}{\eta^2(\tau/2)\eta^2(2\tau)}, \\ \theta_4(\tau)&=&\sum_{k\in\Z}(-1)^kq^{k^2/2}=\frac{\eta^2(\tau/2)}{\eta(\tau)}
	\end{eqnarray*} satisfy the identities
	$\theta_2(\tau)^4+\theta_4(\tau)^4=\theta_3(\tau)^4$ and $\theta_2^4(\tau)+\theta_3^4(\tau)=2E_2(\tau)-E_2(\tau/2)$, where $\displaystyle E_2(\tau)=1-24\sum_{n=1}^\infty \frac{nq^n}{1-q^n}$ with $q=e^{2\pi i \tau}$ is the normalized weight-2 Eisenstein series. Let $\lambda(\tau)=\left(\frac{\theta_2(\tau)}{\theta_3(\tau)}\right)^4$ be the modular $\l$-function. 
	The following three derivative formulas hold (see \cite[Proposition 7]{Zagier-modularform}, \cite[Lemma 1]{LLT}):
	\begin{equation*}
	\frac{d \eta(\tau)}{d\tau}= \eta(\tau)\cdot \frac{2 \pi i }{24} E_2(\tau),
	\end{equation*}
	$$
	\frac{d\theta_4(\tau)}{d\tau}=\theta_4(
	\tau)\cdot \frac{2 \pi i }{24} \cdot \(E_2(\tau/2)-E_2(\tau)\),
	$$\begin{equation*}
	\frac{d\l(\tau)}{d\tau}= \l(\tau)\cdot \pi i \cdot\theta_4^4(\tau).
	\end{equation*}

	We now return to hypergeometric functions.

	Let  $F(t)=\pFq32{\f12&\f12&\f12}{&1&1}{t}$. Then 
	\begin{equation}\label{eq:41}
	\pFq65{\f12&\f12&\f12&\f12&\f12&\f12}{&1&1&1&1&1}{1} =\frac{1}{2\pi i}\oint_{|t|=1}  F(t)F(1/t)\frac{dt}t. 
	\end{equation}

	Using the Clausen formula \eqref{eq:Clausen} with $a=b=\frac14$
	and the following Kummer quadratic formula
	\begin{equation*}
	\pFq21{\f12&\f12}{&1}{z}=(1-z)^{-1/2}\pFq21{\frac14&\frac14}{&1}{\frac{-4z}{(1-z)^2}},
	\end{equation*}
	we have the following identity
	$$
	\pFq32{\f12&\f12&\f12}{&1&1}{\frac{-4z}{(1-z)^2}}=(1-z)
	\pFq21{\f12&\f12}{&1}{z}^2.
	$$
	We now relate the discussion to modular forms by letting $z=\l(2\tau)$ and\begin{equation}\label{eq:t2}
	t_2(\tau)=\frac{-4\l(2\tau)}{(1-\l(2\tau))^2}=-64\frac{\eta(2\tau)^{24}}{\eta(\tau)^{24}},
	\end{equation} which is a Hauptmodul of $\Gamma_0(2)$. It satisfies 
	\begin{align*}
	&t_2\(-\frac1{2\tau} \) = \frac1{t_2(\tau)},\\
	&\frac{d t_2(\tau)}{d \tau}= \frac{4(\l(2\tau)+1)}{(\l(2\tau)-1)^3}\frac{d \l(2\tau)}{d \tau}=\frac{8\pi i \l(2\tau)(\l(2\tau)+1)}{(1-\lambda(2\tau))^3} \theta_4^4(2\tau) .
	\end{align*}

	A classical result says
	$$
	\theta_3(\tau)^4=\pFq{2}{1}{\frac 1{2} &\frac 12}{&1}{\lambda(\tau)}^2, \quad  \theta_4(\tau)^4 =(1- \l(\tau)) \pFq{2}{1}{\frac 1{2} &\frac 12}{&1}{\lambda(\tau)}^2.
	$$
	It follows that
	$$ \pFq32{\f12&\f12&\f12}{&1&1}{t_2(\tau)}= 
	\theta_4^4(2\tau)=\frac{\eta(\tau)^8}{\eta(2\tau)^4} \quad {\rm and}\quad  \pFq32{\f12&\f12&\f12}{&1&1}{1/t_2(\tau)}={16}\tau^2 \frac{\eta(2\tau)^8}{\eta(\tau)^4}.
	$$
	Also, we have
	$\displaystyle
	\frac{1+\l(2\tau)}{1-\l(2\tau)} =1+32\left(\frac{\eta(4\tau)}{\eta(\tau)}\right)^8,
	$
	and thus
	\[
	\pFq32{\f12&\f12&\f12}{&1&1}{t_2(\tau)}\pFq32{\f12&\f12&\f12}{&1&1}{1/t_2(\tau)}\frac{dt_2}{t_2}
	~={ 16}\cdot 2\pi i \tau^2 \cdot f_{8.6.a.a}(\tau/2) { d\tau}, 
	\]where $f_{8.6.a.a}(\tau/2)=\eta^{12}(\tau)+32\eta^4(\tau)\eta(4\tau)^8=-\left(\eta^{12}(\tau)-2\frac{\eta(2\tau)^{24}}{\eta(\tau)^4\eta(4\tau)^8}\right).$ Using \eqref{eq:41}, we get the first claim of Theorem \ref{thm:period}.
	
	\medskip

	Next we demonstrate the appearance of $\fg_{c,f}= f_{8.4.a.a}(\tau)=\eta(2\tau)^4\eta(4\tau)^4$ in the second claim of Theorem \ref{thm:period}  by showing
	\begin{equation*}\label{eq:g} \pFq32{\f12&\f12&\f12}{&1&1}{t_2(\tau)}\pFq43{\f12&\frac54&\f12&\f12}{&\frac14&1&1}{1/t_2(\tau)}\frac{dt_2}{t_2}\sim_e -64\tau f_{8.4.a.a}(\tau) d\tau,
	\end{equation*}
	where the notation $\sim_e$ means the two differentials differ by an exact differential form.
	
	Observe that
	$$\pFq43{\frac54& a&b&c}{&\frac14&d&e}{1/z}=-4 z\frac{d}{dz}\pFq32{ a&b&c}{&d&e}{1/z}+ \pFq32{ a&b&c}{&d&e}{1/z}
	$$
	provided that the series converge. Using 
	\begin{align*}
	\pFq76{\f12&\frac54&\f12&\f12&\f12&\f12&\f12}{&\frac14&1&1&1&1&1}{1}
	=&\frac{1}{2\pi i}\oint_{|z|=1}\pFq43{\frac54& \f 12&\f12 &\f12}{&\frac14&1&1}{1/z} \pFq32{\f 12&\f12 &\f12}{&1&1}z\frac{dz}{z},
	\end{align*}
	and taking $z=t_2$ as \eqref{eq:t2}, we have
	\begin{align*}
	\pFq76{\f12&\frac54&\f12&\f12&\f12&\f12&\f12}{&\frac14&1&1&1&1&1}{1}
	=&\frac{1}{2\pi i}\oint_{|t_2|=1}\pFq43{\frac54& \f 12&\f12 &\f12}{&\frac14&1&1}{1/t_2} F(t_2)\frac{dt_2}{t_2}.\\
	=&\frac{1}{2\pi i}\oint_{|t_2|=1} \(-4t_2 \frac{d}{dt_2} F(1/t_2)+ F(1/t_2)\) F(t_2)\frac{dt_2}{t_2},
	\end{align*}
	where   $F(t)=\pFq32{\f12&\f12&\f12}{&1&1}{t}$ as before.

	In the computation below, we use the following notations: $\theta_i:=\theta_i(2\tau)$, $k:=\frac{\theta_2^2}{\theta_3^2}$, $\l=k^2$. We rewrite $F$-functions as
	$$
	F(t_2(\tau))=\theta_4^4, \quad    F\(\frac 1{t_2(\tau)}\)=4{\tau^2}\theta_3^4 k.
	$$

	Firstly,
	$$
	-4t_2  F(t_2)\frac{d}{dt_2} F(1/t_2)\frac{dt_2}{t_2}=-16\theta_4^4\frac{d(\tau^2 \theta_3^4k)}{d\tau}d\tau,
	$$
	and
	\begin{align*}
	F(1/t_2)F(t_2) \frac{dt_2}{t_2}=&8\pi i \tau^2\theta_3^4\theta_4^4k\cdot \theta_4^4\frac{\l+1}{1-\l} d\tau =8\pi i \tau^2\theta_3^4\theta_4^4k\cdot \left(\theta_4^4+32\frac{\eta^8(4\tau)}{\eta^4(2\tau)}\right)\\
	=& 8\pi i \tau^2\theta_3^4\theta_4^4k\cdot (\theta_4^4+2\theta_2^4) d\tau=8\pi i \tau^2\theta_3^4\theta_4^4k\cdot (\theta_3^4+\theta_2^4)d\tau.
	\end{align*}
	
	Noting 
	$$
	{d(\tau^2 k\theta_3^4\theta_4^4)}=\theta_4^4d(\tau^2 k\theta_3^4)+4\tau^2 k\theta_3^4\theta_4^4\frac{d\theta_4}{\theta_4} 
	$$
	and 
	$$
	\frac{d\theta_4}{\theta_4}=\frac{2 \pi i }{24} \cdot 2\(E_2(\tau)-E_2(2\tau)\) d \tau,
	$$
	one has
	\begin{multline*}
	-4t_2  F(t_2)\frac{d}{dt_2} F(1/t_2)\frac{dt_2}{t_2}\sim_e  -16\(-4\tau^2 k\theta_3^4\theta_4^4\frac{d\theta_4}{\theta_4}\)=  64\tau^2 k\theta_3^4\theta_4^4\(\frac{d\theta_4}{\theta_4}\)\\
	= 64\tau^2 k\theta_3^4\theta_4^4\frac{2\pi i}{24}\cdot 2\(E_2(\tau)-E_2(2\tau)\)d\tau= 8\pi i\tau^2 k\theta_3^4\theta_4^4\cdot \frac{4}{3}\(E_2(\tau)-E_2(2\tau)\)d\tau.
	\end{multline*}

	Hence,
	\begin{eqnarray*}
		&& F(1/t_2)F(t_2) \frac{dt_2}{t_2}-  4t_2  F(t_2)\frac{d}{dt_2} F(1/t_2)\frac{dt_2}{t_2}\\ &\sim_e& 8\pi i \tau^2\theta_3^4\theta_4^4k\cdot  (\theta_3^4+\theta_2^4+\frac{4}{3}\(E_2(\tau)-E_2(2\tau)\))d\tau\\
		&=&8\pi i \tau^2\theta_3^4\theta_4^4k\cdot  (2E_2(2\tau)-E_2(\tau)+\frac{4}{3}\(E_2(\tau)-E_2(2\tau)\))d\tau\\
		&=&8\pi i \tau^2\theta_3^4\theta_4^4k\cdot \frac1{12} (8E_2(2\tau)+4E_2(\tau))d\tau\\
		&=&8\pi i \tau^2\theta_3^4\theta_4^4k\cdot \frac1{12}\frac{24}{2\pi i}\frac{d(\eta(2\tau)^4\eta(\tau)^4)}{\eta(2\tau)^4\eta(\tau)^4}
		=8\tau^2\cdot 4 d(\eta(2\tau)^4\eta(\tau)^4). 
	\end{eqnarray*}
	Thus,
	\begin{eqnarray*}
		&&\pFq76{\frac 12&  \frac 54&\frac 12&\frac 12&\frac 12&\frac 12&\frac 12}{& \frac14&1&1&1&1&1}{1}=\frac{1}{2\pi i}\oint_{ |t_2|=1} 32 \tau^2 d\, \eta(2\tau)^4
		=\frac{32}{2\pi i}\(-2\oint_{ |t_2|=1} \tau \eta(2\tau)^4\eta(\tau)^4d\tau \).
	\end{eqnarray*}

	Finally we consider the path $|t_2|=1$.   As a Hauptmodul of $\G_0(2)$, $t_2(\tau)$ has a simple pole at $0$ and satisfies
	$  t_2\(\frac{i}{\sqrt 2}\)=-1$ and   $t_2\(\frac{1+i}{2}\)=1$, where $i/\sqrt{2}$ is fixed by the Atkin-Lehner involution 
	$w_2$ on $\Gamma_0(2)$. We interpret $|t_2|=1$ as the hyperbolic geodesic from $\frac{1+i}2$ to $\frac{-1+i}2$.   
	Wrapping up the above discussion, we conclude  the two identities stated in Theorem \ref{thm:period}. 
	\footnote{If we denote
		$
		g(\tau)=f_{8.6.a.a}(\tau/2)$, then
		\begin{align*}
		g(\tau\pm 1/2) =&i^{\pm} \left(\left(\frac{\eta(2\tau)^{3}}{\eta(\tau)\eta(4\tau)}\right)^{12}-32\eta(2\tau)^{12}\left(\frac{\eta(4\tau)}{\eta(\tau)}\right)^{4} \right)=i^{\pm}f_{16.6.a.a}\left (\frac{\tau}2\right),
		\end{align*}where $f_{16.6.a.a}$ is the quadratic twist of $f_{8.6.a.a}$ by $\chi_{-1}$.
		It follows that
		\begin{equation}
		_6F_5(HD_1(\f12,\f12))= -4i \left(\int_{i/2}^{i\infty}f_{16.6.a.a}\left (\frac{\tau}2\right)(4\tau^2+1) d\tau\right).
		\end{equation}}

	\subsection{The case $(c,f)=(\f12,\frac13)$} In this subsection, we will prove the following identity
	
	$$
	\frac1{\pi}  \pFq 65{\f12&\f12&\f12&\f12&\frac13&\frac23}{&\frac56&\frac76&1&1&1}{1}=6i\oint_{ |t_3|=1}  \(\frac13+\tau+\tau^2\)  \cdot(f_{4.6.a.a}(\tau /2)-27f_{4.6.a.a}(3\tau/2)) d\tau,$$ where 
	\begin{equation}\label{eq:t3}
	\displaystyle     t_3(\tau)=4\(\frac{1}{3\sqrt 3}\frac{\eta^6(\tau)}{\eta^6(3\tau)}+3\sqrt 3\frac{\eta^6(3\tau)}{\eta^6(\tau)}\)^{-2}
	\end{equation}
	is a Hauptmodul of $\G_0^+(3)$  taking values $0$, $1$, $\infty$ at the elliptic points of order $6$, $2$, and at the cusp (of order $\infty$), respectively.  To {proceed}, 
	we use
	$$
	\pFq 65{\f12&\f12&\f12&\f12&\frac13&\frac23}{&\frac56&\frac76&1&1&1}{1}=\frac{1}{2\pi i}\oint_{|t|=1} \pFq32{\f12&\f12&\f12}{&\frac56&\frac76}{1/t}\pFq32{\f12&\frac13&\frac23}{&1&1}{t}\frac{dt}t
	$$

	The second  $_3F_2$ term can be handled by the Clausen formula \eqref{eq:Clausen}
	\begin{equation*}
	\pFq32{\f12&\frac13&\frac23}{&1&1}{4z(1-z)}=\pFq21{\frac13&\frac23}{&1}{z}^2.
	\end{equation*}
	
	To deal with the first one, we use Property (P.2).  
	By a straightforward verification, one has
	\begin{align*}
	\frac{\sqrt 3}{\pi}&(-4z(1-z))^{-1/2}\pFq32{\f12&\f12&\f12}{&\frac76&\frac56}{\frac1{4z(1-z)}}\\
	=& \frac13\pFq21{\frac13&\frac23}{&1}{z}^2+\frac i{\sqrt{3}}\pFq21{\frac13&\frac23}{&1}{z}\pFq21{\frac13&\frac23}{&1}{1-z}-\frac 13\pFq21{\frac13&\frac23}{&1}{1-z}^2.
	\end{align*}
	
	When $z=\(1+\eta(\tau)^{12}/27\eta(3\tau)^{12}\)^{-1}$,  we can express a point on the upper-half plane by $$\tau=\frac i{\sqrt 3} \pFq21{\frac13&\frac23}{&1}{1-z}/\pFq21{\frac13&\frac23}{&1}{z} .$$ Such results are part of Ramanujan's alternative theory of elliptic functions. The details can be found in \cite{BBG-Ramanujan} by Berndt, Bhargava and Garvan for example. 
	The previous equation becomes
	$$
	\frac{\sqrt 3}{\pi}(-4z(1-z))^{-1/2}\pFq32{\f12&\f12&\f12}{&\frac76&\frac56}{\frac1{4z(1-z)}}=\(\frac13+\tau+\tau^2\)\pFq32{\f12&\frac13&\frac23}{&1&1}{4z(1-z)}.
	$$
	Note that $\(1+\eta(\tau)^{12}/27\eta(3\tau)^{12}\)^{-1}$ is the modular function $s^3$ in the work  \cite{BBG}, where $s=c/a$ 
	and $a,c$ are two weight-1 modular forms  recalled in \cite[\S 2]{LLT}.
	From the formula \eqref{eq:t3},
	$$
	\frac{dt_3}{t_3}=\frac{3(2s^3-1)}{s(s^3-1)}ds=-2\pi i\cdot a^2(2s^3-1)d\tau
	$$
	and
	$$
	\pFq32{\f12&\frac13&\frac23}{&1&1}{t_3}=\pFq21{\frac13&\frac23}{&1}{s^3}^2=a^2.
	$$
	We conclude that
	\begin{align*}
	\pFq32{\f12&\f12&\f12}{&\frac56&\frac76}{1/t_3}&\pFq32{\f12&\frac13&\frac23}{&1&1}{t_3}\frac{dt_3}{t_3}\\
	=& \(\frac13+\tau+\tau^2\) \frac{\pi}{\sqrt 3}(-t_3)^{1/2} \pFq32{\f12&\frac13&\frac23}{&1&1}{t_3}^2\frac{dt_3}{t_3}\\
	=&\(\frac13+\tau+\tau^2\) \frac{4\pi^2}{\sqrt 3}(s^3(1-s^3))^{1/2}a^6(2s^3-1)d\tau\\
	=&-12\pi^2\cdot \(\frac13+\tau+\tau^2\)  \cdot(f_{4.6.a.a}(\tau /2)-27f_{4.6.a.a}(3\tau/2)) d\tau.
	\end{align*}
	
	
	\subsection{The case $(c,f)=(\f12,\frac16)$}In this case the unit root  of the local zeta-function for $HD_1$ may not be unique.  Despite it, we still have the following  conjectural congruence in \cite{Long18} for any prime $p\ge 5$, 
	\begin{equation*}
	\pFq65{\f12&\f12&\f12&\f12&\frac16&\frac56}{&1&1&1&\frac43&\frac23}{1}_{p-1}\overset ? \equiv a_p(f_{8.4.a.a}) \mod p^3,
	\end{equation*}  which is verified for  the  primes between 4 and 500, showing the exponent 3 sharp.
	
	The archimedean counterpart is  
	\begin{equation}\label{eq:F1/2-1/6}
	\pFq65{\f12&\f12&\f12&\f12&\frac16&\frac56}{&1&1&1&\frac43&\frac23}{1}
	={16}
	\oint_{ |t_2|=1}\left( f_{8.4.a.a}\left (\frac{\tau}2\right)+27f_{8.4.a.a}\left (\frac{3\tau}2\right)\right ) d\tau,
	\end{equation} where $t_2$ is  given in \eqref{eq:t2}. To prove \eqref{eq:F1/2-1/6}, we start  with the following Lemma. 
	
	\begin{lemma}\label{lem:233-3F2}
		The following identity holds near $x=0$ when both sides converge: 
		\begin{equation*}
		\pFq32{\f12&\frac16&\frac56}{&\frac{4}{3}&\frac2{3}}{\frac{x(x+4)^3}{4(2x-1)^3}}=\frac{4(1-2x)^{1/2}}{4+x}.
		\end{equation*}
	\end{lemma}
	\begin{proof}
		Letting $a=\frac14,b=\frac7{12}$ in \eqref{eq:gen-Clausen}, we get
		\begin{equation*}
		\pFq21{\frac1{4}&\frac7{12}}{&\frac4{3}}{x}\pFq21{\frac1{4}&\frac{-1}{12}}{&\frac{2}{3}}{x}=\pFq32{\f12&\frac16&\frac56}{&\frac{4}{3}&\frac2{3}}{x}.
		\end{equation*}By (9.25) of \cite{Win3X}
		\begin{equation*}
		\pFq21{\frac1{4}&\frac{-1}{12}}{&\frac{2}{3}}{f(x)} =\frac1{(1-2x)^{1/4}},
		\end{equation*}where $\displaystyle f(x)=\frac{x(x+4)^3}{4(2x-1)^3}$ and $f(-\frac 2x)=\frac1{f(x)}$. Using (3.10) of \cite{Win3X}, we get 
		\begin{equation*}
		\pFq21{\frac14&\frac{-1}{12}}{&\frac23}{f(x)}=\frac1{\sqrt{24}}(-f(x))^{-1/4}\pFq21{\frac14&\frac7{12}}{&\frac43}{\frac1{f(x)}}+\sqrt{\frac23}(-f(x))^{1/12}\pFq21{\frac14&\frac{-1}{12}}{&\frac23}{\frac1{f(x)}}
		\end{equation*} from which we deduce 
		
		\begin{equation*}
		\pFq21{\frac1{4}&\frac{7}{12}}{&\frac{4}{3}}{f(x)} =\frac{4(1-2x)^{3/4}}{4+x}.
		\end{equation*}
		The claim of the Lemma follows.
	\end{proof}

	Let $x$ be  a  Hauptmodul  $\displaystyle -\frac 12\frac{\eta(\tau)^5\eta(3\tau)}{\eta(2\tau)\eta( 6\tau)^5}-4$ on $\G_0(6)$ and $\displaystyle f(x)=\frac{x(x+4)^3}{4(2x-1)^3}$ as above. Then $f(x)=1/t_2$ and hence $f\(\frac{-2}x\)=t_2$.
	By  Lemma \ref{lem:233-3F2},
	\begin{align*}
	&\pFq32{\f12&\f12&\f12}{&1&1}{1/t_2}\pFq32{\f12&\frac16&\frac56}{&\frac43&\frac23}{t_2}\frac{dt_2}{t_2}
	=\pFq32{\f12&\f12&\f12}{&1&1}{1/t_2}\pFq32{\f12&\frac16&\frac56}{&\frac{4}{3}&\frac2{3}}{f\(\frac{-2}x\)}\frac{dt_2}{t_2}\\
	=&-16\tau^2 \frac{\eta(2\tau)^8}{\eta(\tau)^4}\cdot \(\frac{x+4}x\)^{1/2}\cdot \frac{4(x^2-10x-2)}{(1-2x)^2}\frac{ d(-2(x+4))}{-2(x+4)}\\=&32\pi i \left (f_{8.4.a.a}\left (\frac{\tau}2\right)+27f_{8.4.a.a}\left (\frac{3\tau}2\right)\right ) \tau^2 d\tau,
	\end{align*} 
	where
	\begin{align*}
	\frac{ d(-2(x+4))}{-2(x+4)}
	=\frac{2\pi i}{24}\(5E_2(\tau)+3E_2(3\tau)-30E_2(6\tau)-2E_2(2\tau)\){d\tau}.
	\end{align*}
	
	\subsection{Other numerical conjectures}  Our numerical data for primes up to 500 suggest the following supercongruences for  
	primes $p>11$, 
	with the corresponding $p$ powers  sharp for generic primes $p$ in each case: 
	$$
	(c,f)=(\f12,\frac13):\quad \pFq76{\f12&\frac54&\f12&\f12&\f12&\frac13&\frac23}{&\frac14&\frac76&\frac16&1&1&1}{1}_{p-1} \overset ?\equiv a_p(f_{12.4.a.a})\mod p^3.
	$$
	
	$$
	(c,f)=(\f12,\frac16):\quad \pFq76{\f12&\frac54&\f12&\f12&\f12&\frac16&\frac56}{&\frac14&\frac43&\frac23&1&1&1}{1}_{p-1}\overset ? \equiv a_p(f_{24.4.a.a})\mod p^3.
	$$
	
	$$
	(c,f)=(\frac13,\frac13):\quad   p^2\cdot \pFq 65{\f12&\frac13&\frac23&\f12&\frac13&\frac23}{&\frac56&\frac76&1&\frac56&\frac76}{1}_{p-1}\overset ?\equiv a_p(f_{6.6.a.a}) \mod p^4.
	$$
	
	$$
	(c,f)=(\frac13,\frac13):\quad   p\cdot \pFq 76{\f12&\frac54&\frac13&\frac23&\f12&\frac13&\frac23}{&\frac14&\frac56&\frac76&1&\frac56&\frac76}{1}_{p-1}\overset ?\equiv a_p(f_{18.4.a.a}) \mod p^3.
	$$
	
	$$
	(c,f)=(\frac15,\frac25):\quad   p\cdot \pFq 65{\f12&\frac15&\frac25&\f12&\frac35&\frac45}{&\frac{11}{10}&\frac9{10}&1&\frac{13}{10}&\frac7{10}}{1}_{p-1}\overset ?\equiv a_p(f_{10.4.a.a}) \mod p^3.
	$$
	$$
	(c,f)=(\frac15,\frac25):\quad   p\cdot \pFq 76{\f12&\frac54&\frac15&\frac25&\f12&\frac45&\frac45}{&\frac14&\frac{11}{10}&\frac9{10}&1&\frac{13}{10}&\frac7{10}}{1} _{p-1}\overset ?\equiv a_p(f_{50.4.a.d}) \mod p^3.
	$$
	\section{Acknowledgment}
	The research of Li is partially supported by Simons Foundation grant \# 355798. Long was supported in part by NSF DMS \# 1602047 and  the paper was written during her sabbatical leave  in Fall 2020.   The first two authors are grateful for their visits to the National Center for Theoretical Sciences in 2019. Moreover, the authors would like to thank Dr. Siu-Hung Ng, Wadim Zudilin,  and the anonymous referees  for helpful suggestions.


	
	
	\address{$^\dag$~Department of Mathematics, Pennsylvania State University, University Park, PA 16802, USA}
	
	\email{\href{mailto:wli@math.psu.edu}{wli@math.psu.edu}}, {\url{http://www.math.psu.edu/wli/}}
	\medskip
	
	\address{$^\ddag$~Department of Mathematics, Louisiana State University, Baton Rouge, LA 70803, USA}
	
	\email{\href{mailto:llong@lsu.edu}{llong@lsu.edu}, {\url{http://www.math.lsu.edu/~llong/}
			
			\href{mailto:tu@math.lsu.edu}{ftu@lsu.edu}}, \url{https://sites.google.com/view/ft-tu/}}
\end{document}